\documentclass[a4paper,12pt]{article}

\usepackage{enumitem}
\usepackage[utf8]{inputenc}
\usepackage[english]{babel}
\usepackage{csquotes}
\usepackage{graphicx}
\usepackage{amsmath, amsthm, amssymb}
\usepackage{marvosym}

\newcommand{\hhookrightarrow}{\lhook\mkern-3mu\relbar\mkern-12mu\hookrightarrow}

\newtheorem{theorem}{Theorem}[section]
\theoremstyle{definition} 
\newtheorem{definition}[theorem]{Definition} 
\theoremstyle{definition} 
\newtheorem{lemma}[theorem]{Lemma} 
\theoremstyle{definition} 
\newtheorem{remark}[theorem]{Remark}
\theoremstyle{definition} 
 
\theoremstyle{definition}
 
\theoremstyle{definition} 
\newtheorem{assumption}[theorem]{Assumption} 
\theoremstyle{definition} 

\title{Global existence and Hadamard differentiability of hysteresis-reaction-diffusion systems}
\author{Christian Münch\footnote{Department of Mathematics - M6, Technical University of Munich, Boltzmannstr. 3, 85747 Garching, Germany. christian.muench@ma.tum.de}}

\usepackage{hyperref}

\begin{document}
\maketitle

%\pagenumbering{roman}
%\pagestyle{headings}
%%%%% Zusammenfassung in deutscher Sprache
%
%\newpage
%\tableofcontents
%\newpage
%
%%%%% Page numbering restarts here
%\pagenumbering{arabic}
%\pagestyle{headings}

\pagenumbering{arabic}  

\begin{abstract}
	
	We consider a class of semilinear parabolic evolution equations subject to a hysteresis operator 
	and a Bochner-Lebesgue integrable source term.
	The underlying spatial domain is allowed to have a very general boundary.
	In the first part of the paper, we apply semigroup theory to prove well-posedness and boundedness of the solution operator.
	Rate independence in reaction-diffusion systems complicates the analysis, since the reaction term acts no longer local in time.
	This demands careful estimates when working with semigroup methods.
	In the second part, we show Lipschitz continuity and Hadamard differentiability of the solution operator.
	We use fixed point arguments to derive a representation for the derivative in terms of the evolution system.
	Finally, we apply our results to an optimal control problem in which the source term acts as a control function
	and show existence of an optimal solution.
%	We introduce a class of semilinear parabolic evolution equations with hysteresis
%	in the non-linearity and with an additional Bochner-Lebesgue integrable source term. 
%	The underlying domain is allowed to have a very general boundary.
%	In the first part of the paper, we analyze the problem as a function in the source term. 
%	Semigroup theory is applied to prove well-posedness and boundedness of the corresponding solution operator.
%	In the second part, we show Lipschitz continuity and Hadamard differentiability of the solution operator. 
%	We apply fixed point arguments to derive a representation for the derivative in terms of evolution systems.
%	In the end of the paper, we apply our results to an optimal control problem in which the source term acts a control function.
%	We prove existence of an optimal solution.
%	Rate independence in the reaction term of a reaction-diffusion system complicates the whole analysis, since the reaction term acts no longer local in time.
%	This demands additional work when working with semigroup methods.

%	
%	Rate independence in the reaction term of a reaction-diffusion system complicates the analysis of the solution operator. Additional difficulties arise if the underlying domain is non-smooth.
%	
%	We introduce a class of semilinear parabolic evolution equations with hysteresis in the non-linearity and prove well-posedness. 
%	Moreover, we show Hadamard differentiability of the solution operator in the source term.
%	Afterwards, we apply our results to an optimal control problem.
\end{abstract}

\begin{center}

MSC subject class: 47J40, 35K51
\end{center}

\section{Introduction}

In this paper we analyze semilinear parabolic evolution equations of the form
\begin{alignat}{2}
\frac{d}{dt}y(t) + (A_p y)(t) &= (F[y])(t) + u(t)\ && \text{ in } X \text{ for } t>0,\label{state_eqation_abstract}\\
y(0)&=0\in X.\notag
\end{alignat}
In this context $X$ is a product of dual spaces and $A_p$ is an unbounded operator on $X$. The non-linearity $F$ is a Nemytski operator, i.e. $(F[y])(t)=f(y(t),\mathcal{W}[S y](t))$.
$S$ is a linear operator which transforms the vector valued function $y$ into a scalar valued map. $\mathcal{W}$ is a scalar stop operator. 
One way to represent the value of $z=\mathcal{W}[v]$ is as the unique solution of the variational inequality 
\begin{alignat}{2}
(\dot{z}(t)-\dot{v}(t))(z(t)-\xi) &\leq 0 &&\ \text{for } \xi\in [a,b] \text{ and } t\in (0,T),\label{stop1}\\
z(t)&\in [a,b] &&\ \text{for } t\in[0,T],\label{stop2}\\
z(0)&=z_0\label{stop3}
\end{alignat} 
\cite{brokate2013optimal}.
The forcing term $u\in \mathrm{L}^q(J_T;X)$ may for example serve as a control. 
Our choice for the notation in equation~\eqref{state_eqation_abstract} is motivated by the application of our results to optimal control theory.

The major focus of this paper are well-posedness of \eqref{state_eqation_abstract} and Hadamard directional differentiability of the solution operator $G$ which maps each $u$ to the corresponding solution $y$ of \eqref{state_eqation_abstract}.

%As the first major result in this paper we show that for every $u\in \mathrm{L}^q(J_T;X)$ with $\frac{1}{1-\alpha} < q \leq \infty$ problem \eqref{state_eqation_abstract} has a unique mild solution $y\in \mathrm{C}(\overline{J_T};X^\alpha)$. The space $X^\alpha$ is a fractional power space.
General semilinear parabolic problems with Lipschitz continuous non-linearities $f(t,y(t))$ and with a forcing term $u(t)$ which is Bochner-Lebesgue integrable have, for instance, been analyzed in \cite{lunardi}. Differentiability of the solution mapping is discussed in \cite{meyeroptimal}. 
Abstract evolution equations with (locally) Lipschitz continuous right-hand sides $f(t,y(t))$ and without an additional forcing term are for instance treated in \cite{henry,pazy} and \cite{lunardi}.
In these cases, the non-linearity $f$ is local in time.

The main novelty of this paper comes from the hysteresis $\mathcal{W}$, which is non-local in time. This adds a new challenge to the question of well-posedness since $\mathcal{W}[Sy](t)$ depends not only on $t$ but on the whole time history of $y$ in $[0,t]$. 
Furthermore, $\mathcal{W}$ is non-smooth so that differentiability of the solution operator to \eqref{state_eqation_abstract} is not clear at all.
Because we can not expect Fréchet differentiability \cite{Brokate_weak_diff}, we turn to the concept of Hadamard directional differentiability.

This work is organized as follows.

%In Section~\ref{Sec:Results_Lit} , with $1<p<\infty$, we introduce the space $\mathbb{W}_{\Gamma_D}^{1,p'}(\Omega)$ of vector valued functions with homogeneous boundary conditions on a part of $\partial\Omega$ and its dual space $\mathbb{W}_{\Gamma_D}^{-1,p}(\Omega)$.
%
%Then we properly define the unbounded
%operator
%\[
%A_p:=-\nabla\cdot D\nabla : \mathrm{dom}(A_p)\subset \mathbb{W}_{\Gamma_D}^{-1,p}(\Omega) \rightarrow \mathbb{W}_{\Gamma_D}^{-1,p}(\Omega)
%\]
%which is induced by
%\[
%\langle -\nabla\cdot D\nabla u,v \rangle_{\mathbb{W}_{\Gamma_D}^{1,2}(\Omega)} =\int_\Omega D\nabla u \cdot \nabla v\, dx
%\]
%for $u,v\in \mathbb{W}_{\Gamma_D}^{1,2}(\Omega)$.
In Section~\ref{Sec:Results_Lit} we collect results from the literature and state the main assumption.
We do not consider product spaces of $\mathrm{L}^p(\Omega)$-functions for $X$ because we include very general domains $\Omega$. The right side of equation~\eqref{state_eqation_abstract} therefore takes its values only in a product of dual spaces. It is not easy to find a fully elaborated description of the functional setup for our problem. We do our best to provide a precise framework which includes all the required results.

%In Subsection\eqref{SUBSEC:Results_Lit} we combine theorems from \cite{rehbergsystems,amann_measures,henry,meyeroptimal} and \cite{squareroot_problem} to obtain the needed regularity properties for $A_p$ and define fractional power spaces.
%Moreover, we collect the properties of the stop operator which we need.
%
%In Section~\ref{Sec:Ass} we state the assumptions for the rest of the work and introduce some notation.

%In the control problem we consider either
%distributed controls
%\begin{align*}
%U_1:= \mathrm{L}^2 \left((0,T);[\mathrm{L}^2(\Omega)]^m\right)
%\end{align*}
%or Neumann boundary controls
%\begin{align*}
%U_2:= \mathrm{L}^2 \left((0,T); \prod_{i=1}^m \mathrm{L}^2(\Gamma_{N_i},\mathcal{H}_{d-1})\right).
%\end{align*}
%
%We define embeddings
%$B_i : U_i \rightarrow \mathbb{W}_{\Gamma_D}^{-1,p}(\Omega)$ for each space of control functions.

In Section~\ref{Sec:Well_posed} we show well-posedness of equation~\eqref{state_eqation_abstract} with
$u\in \mathrm{L}^q((0,T);X)$. Theorem~\ref{Thm:state_equ_wellposed} is the first main result of this work.

%With $u\in U_i$, the latter is given by the following system of evolution equations:
%\begin{alignat}{2}
%\dot{y}(t) +A_p y(t) &= f(y(t),z(t))  + B_i u(t)&&\  \text{in } \mathbb{W}_{\Gamma_D}^{-1,p}(\Omega) \text{ for }t\in (0,T), \label{state_equ_y}\\
%y(0)&= 0 &&\ \text{in } \mathbb{W}_{\Gamma_D}^{-1,p}(\Omega),\notag\\
%(\dot{z}(t)-S\dot{y}(t))(z(t)-\xi) &\leq 0 &&\ \text{for } \xi\in [a,b] \text{ and } t\in (0,T),\label{state_equ_z}\\
%z(t)&\in [a,b] &&\ \text{for } t\in[0,T],\notag\\
%z(0)&=z_0\in [a,b]\notag.
%\end{alignat} 
%
%In this context $S$ transforms $y$ to a scalar valued function and $z=\mathcal{W}[Sy]$ is the output of the stop operator defined on the interval $[a,b]$ with initial value $z_0$ and input $Sy$.
%$y$ takes values in a distribution space.

After defining Hadamard directional differentiability, Section~\ref{Sec:Hadam_diff} contains a proof that the solution operator for \eqref{state_eqation_abstract} in $u$ has this property. Theorem~\ref{Thm:Solution_op_Hadamard} is our second main result.

%If $f=B_1f_\Omega(y(t),z(t))$ for some function $f_\Omega$ which has values in $U_1$ and if the reaction term is regular enough and if
%
% in time and space, $\dot{y}$ takes values in $\mathrm{L}^{\infty}(\Omega)$ so that \eqref{state_equ_y} is the weak solution of the problem 
%
%\begin{alignat*}{2}
%\dot{y}(t) - [\nabla\cdot D\nabla y(t)] &= f_{\Omega}(y(t),z(t)) + u&&\ \text{ in } \mathrm{L}^p(\Omega,\mathbb{R}^m),\ t>0, \\
%n\cdot D_i \nabla y_i(t) &= 0 &&\ \text{ in } \mathrm{L}^r(\Gamma_i),\ t>0, i=1,...,m, \\
%y(t) &= 0&& \text{ on } \mathrm{\Gamma_D},\ t\in [0,T], \\
%y(0)&= y_0&&\in \mathbb{W}_{\mathrm{\Gamma_D}}^{1,p}(\Omega),
%\end{alignat*} 
%in the case of $U_1$
%or
%\begin{alignat*}{2}
%\dot{y}(t) - [\nabla\cdot D\nabla y(t)] &= f_{\Omega}(y(t),z(t))&&\ \text{ in } \mathrm{L}^p(\Omega,\mathbb{R}^m),\ t>0, \\
%n\cdot D_i \nabla y_i(t) &= u_i(t) &&\ \text{ in } \mathrm{L}^r(\Gamma_i),\ t>0, i=1,...,m, \\
%y(t) &= 0&& \text{ on } \mathrm{\Gamma_D},\ t\in [0,T], \\
%y(0)&= y_0&&\in \mathbb{W}_{\mathrm{\Gamma_D}}^{1,p}(\Omega)
%\end{alignat*} 
%in the case of $U_2(0,T)$.
%
%We will only work with the state equation as stated in \eqref{state_equ_y}.

%We consider the following optimal control problem for $i\in\{1,2\}$:
%\begin{align}
%\min_{u \in U_i} J(y,u) &:= \frac{1}{2}  \|y-y_d\|_{U_1}^2+ \frac{\kappa}{2}\| u \|_{U_i}^2 \label{opt_control_ control_problem}
%\end{align} 
%subject to \eqref{state_equ_y}, \eqref{state_equ_z}.
%
%Existence of an optimal control for \eqref{opt_control_ control_problem} is shown in Section~\ref{Sec:Optimal_contr}.\\

In Section~\ref{Sec:Optimal_contr} we apply Theorem~\ref{Thm:state_equ_wellposed} and Theorem~\ref{Thm:Solution_op_Hadamard} to an optimal control problem where the state equation takes the form of \eqref{state_eqation_abstract}.
Existence of an optimal control is shown in Theorem~\ref{Thm:exist_opt_control}.

The results from Section~\ref{Sec:Well_posed} and Section~\ref{Sec:Hadam_diff} are also valid if $A_p$ is replaced by a more general sectorial operator $\mathrm{T}_p$ which does not necessarily have to satisfy maximal parabolic Sobolev regularity. In this case $y$ is a continuous function with values in a fractional power space.
Equation \eqref{state_eqation_abstract} has to be interpreted in the sense of mild solutions then.
%In a forthcoming paper we will construct such operators and prove resolvent estimates in order to show that they are sectorial.
The scalar stop operator $\mathcal{W}$ can be replaced by a general hysteresis operator with appropriate properties, cf. Remark~\ref{Rem:generalize_hyst_op}.
In this paper, we focus on the operators $A_p$ and $\mathcal{W}$ in order to give an illustration right away.\\

%In another forthcoming paper we will also derive an adjoint system and optimality conditions for problem \eqref{opt_control_ control_problem}.
%The differences between the control problem for $U_1$ and $U_2$ will become obvious during this analysis.
%
%By density of $B_1$ we will be able to improve the optimality conditions for the general problem with $B_i$. We can also show uniqueness of the adjoint system for the case of distributed controls.\\
%Our state space will have the form
%\begin{align*}
%Y_s:= \mathrm{W}^{1,s}((0,T),\mathbb{W}_{\Gamma_D}^{-1,p}(\Omega))\cap \mathrm{L}^{s}((0,T),\mathrm{dom}(A_p))
%\end{align*}
%for some $2\leq s < \infty$.
%
%We will also write
%
%\begin{align*}
%Y_{s,0,y_0}:= \{y\in Y_s,\ y(0)=y_0\}
%\end{align*}
%and 
%\begin{align*}
%Y_{s,0}:= \{y\in Y_s,\ y(0)=0\}.
%\end{align*}

We write $\mathcal{L}(X,Y)$ for the space of linear operators between spaces $X$ and $Y$ and $\mathcal{L}(X)$ for the space of linear operators on $X$.
We also abbreviate the duality in $X$ by
\[
\langle x,y \rangle_{X^*,X} =
\langle x,y \rangle_{X}.
\]

\section{Preliminaries and assumptions}\label{Sec:Results_Lit}
\subsection{Sobolev spaces including homogeneous Dirichlet boundary conditions}
The setting and the theory of this section is strongly based on results from \cite{rehbergsystems}. We recall several definitions, results and assumptions from this work.
All Sobolev spaces are defined on a bounded domain $\Omega\subset \mathbb{R}^d$ with $d\geq 2$. The boundary regularity is defined in Assumption~\ref{Ass:domain}.

We only consider real valued functions.

For each component $j\in \{1,\cdots, m\}$ of the space of vector valued functions, see Defnition~\ref{Def:Sobolev_space}, the boundary $\partial\Omega$ is decomposed into the corresponding Dirichlet part $\Gamma_{D_j}$ and the Neumann boundary $\Gamma_{N_j}:=\partial\Omega\backslash \Gamma_{D_j}$, see Assumption~\ref{Ass:domain}. The cases $\Gamma_{D_j}=\emptyset$ and $\Gamma_{D_j}=\partial\Omega$ are not excluded \cite[Comment after Definition 2.4]{rehbergsystems} and \cite[Remark 2.2 (iii)]{squareroot_problem}.
The assumed condition on $\Gamma_{D_j}$ requires the definition of an $I$-set where $I\in (0,d]$ \cite[Definition 2.1]{rehbergsystems}.
\begin{definition}
	For $0<I\leq d$ and a closed set $M\subset \mathbb{R}^d$ let $\rho$ denote the restriction of the $I$-dimensional Hausdorff measure $\mathcal{H}_I$ to $M$. Then we call $M$ an $I$-set if there are constants $c_1,c_2>0$ such that
	\begin{align*}
	c_1 r^I \leq \rho\left( B_{\mathbb{R}^d}(x,r)\cap M \right) \leq c_2 r^I
	\end{align*}
for all $x$ in $M$ and $r\in ]0,1[$.
\end{definition}

The assumption on the domain in our setting is the following \cite[Assumption 2.3]{rehbergsystems}:
\begin{assumption}\label{Ass:domain}
	The domain $\Omega\subset\mathbb{R}^d$ is bounded and $\overline{\Omega}$ is a $d$-set.
	
	For $j\in \{1,\cdots, m\}$ the Neumann boundary part $\Gamma_{N_j}\subset \partial\Omega$ is open and $\Gamma_{D_j}=\partial\Omega\backslash{\Gamma_{N_j}}$ is a $(d-1)$-set.
\end{assumption}

\begin{remark}
	As already mentioned in the beginning of this section, note that the cases $\Gamma_{D_j}=\emptyset$ and $\Gamma_{D_j}=\partial\Omega$ are not excluded \cite[Comment after Definition 2.4]{rehbergsystems} and \cite[Remark 2.2 (iii)]{squareroot_problem}.
	Assumption~\ref{Ass:domain} allows for very general domains. For example, $\Omega$ may be a Lipschitz domain and for $j\in \{1,\cdots, m\}$, $\Gamma_{D_j}$ can be a $(d-1)$-dimensional manifold.
\end{remark}

In the same manner as in \cite[Definition 2.4]{rehbergsystems} we define Sobolev spaces which include the Dirichlet boundary conditions for our state equation.

\begin{definition}\label{Def:Sobolev_space}
	Let $U\subset \mathbb{R}^d$ be a domain and $p\in [1,\infty)$. 
	\begin{itemize}
		\item $\mathrm{W}^{1,p}(U)$ denotes the usual Sobolev space of functions $\psi\in \mathrm{L}^p(U)$ whose weak partial derivatives exist in $\mathrm{L}^p(U)$.
		The norm in $\mathrm{W}^{1,p}(U)$ is 
		\[\Vert \psi \Vert_{\mathrm{W}^{1,p}(U)}= \left( \int\limits_U \left( \vert \psi \vert^2 + \sum\limits_{j=1}^d \left\vert \frac{\partial\psi}{\partial x_j} \right\vert^2\right)^{\frac{p}{2}} \,dx\right)^{\frac{1}{p}}.\]
		
		\item For a closed subset $M$ of $\overline{U}$ we define
		\[
		\mathrm{C}_M^\infty(U):=\lbrace \psi|_U:\ \psi\in\mathrm{C}_0^\infty(\mathbb{R}^d),\ \mathrm{supp}(\psi)\cap \mathrm{M}=\emptyset \rbrace
		\]
		and denote by $\mathrm{W}_\mathrm{M}^{1,p}(U)$ the closure of $\mathrm{C}_M^\infty(U)$ in $\mathrm{W}^{1,p}(U)$.
		\item
		For $p>1$ we write $p'$ for the Hölder conjugate of $p$.
		
		The dual space $\left[\mathrm{W}_\mathrm{M}^{1,p'}(U)\right]^*$ of $\mathrm{W}_\mathrm{M}^{1,p'}(U)$ is called $\mathrm{W}_\mathrm{M}^{-1,p}(U)$.
	\end{itemize}
\end{definition}
\begin{remark}
	We stick to the norm which is used in \cite{rehbergsystems} which differs from the usual norm in Sobolev spaces. One reason for this choice is that it simplifies estimates concerning the duality between $\mathrm{W}_\mathrm{M}^{1,p}(U)$ and $\mathrm{W}_\mathrm{M}^{1,p'}(U)$. We may identify a function $\phi\in \mathrm{W}_\mathrm{M}^{1,p}(U)$ with an element in $\mathrm{W}_\mathrm{M}^{-1,p}(U)$ since for any $\psi\in \mathrm{W}_\mathrm{M}^{1,p'}(U)$ the Cauchy Schwarz inequality together with Hölder's inequality yields
	\begin{align*}
		&\int\limits_U 
				\left( 
					\phi \psi + \sum\limits_{j=1}^d \frac{\partial\phi}{\partial x_j} \frac{\partial\psi}{\partial x_j} 
				\right) 
		\,dx\\
		&\leq \int\limits_U 
			\left(
			\vert \phi \vert^2 + \sum\limits_{j=1}^d \left\vert \frac{\partial\phi}{\partial x_j} \right\vert^2
			\right)^{\frac{1}{2}}
			\left(
				\vert \psi \vert^2 + \sum\limits_{j=1}^d \left\vert \frac{\partial\psi}{\partial x_j} \right\vert^2
			\right)^{\frac{1}{2}}
		  \,dx\\
		&\leq \left( \int\limits_U \left( \vert \phi \vert^2 + \sum\limits_{j=1}^d \left\vert \frac{\partial\phi}{\partial x_j} \right\vert^2\right)^{\frac{p}{2}} \,dx\right)^{\frac{1}{p}}
		\left( \int\limits_U \left( \vert \psi \vert^2 + \sum\limits_{j=1}^d \left\vert \frac{\partial\psi}{\partial x_j} \right\vert^2\right)^{\frac{p'}{2}} \,dx\right)^{\frac{1}{p'}}.
	\end{align*}
\end{remark}

We need the following assumption for each of the $m$ components \cite[Assumption 4.11]{rehbergsystems}:
\begin{assumption}\label{Ass:existence_extension}
	In the setting of Assumption \ref{Ass:domain} we suppose for all $j\in \{1,\cdots,m\}$ and any $x\in \overline{\Gamma_{N_j}}$ that there is an open neighborhood $U_x$ of $x$ and a bi-Lipschitz mapping $\phi_x$ from $U_x$ onto a cube in $\mathbb{R}^d$ such that $\phi_x(\Omega\cap U_x)$ equals the lower half of the cube and such that $\partial\Omega\cap U_x$ is mapped onto the top surface of the lower half cube.
\end{assumption}

\begin{remark}\label{Rem:embedding_w1p_lp}
	Assumption~\ref{Ass:existence_extension} has the following consequences:
	\begin{enumerate}
		\item
		Firstly, Assumption~\ref{Ass:existence_extension} is needed in order to assure the existence of continuous extension operators from $\mathrm{W}_{\Gamma_{D_j}}^{1,p}(\Omega)$ to $\mathrm{W}_{\Gamma_{D_j}}^{1,p}(\mathbb{R}^d)$ for all $j\in\{1,\cdots,m\}$ and $p\in (1,\infty)$. This in turn is required in \cite[Section 3]{rehbergsystems} to establish interpolation properties between the spaces $\{\mathrm{W}_{\Gamma_{D_j}}^{1,p}(\Omega)\}_{p\in (1,\infty)}$ for fixed $j\in\{1,\cdots,m\}$.
		Secondly, the assumption is used in \cite[Section 5]{rehbergsystems} to prove elliptic and parabolic regularity results, see Theorem~\ref{Thm:elliptic_regularity_for_systems} below.
		\item 
			Under Assumption~\ref{Ass:existence_extension} it can be shown that the embeddings
			$\mathrm{W}_{\Gamma_{D_j}}^{1,p}(\Omega)\hookrightarrow \mathrm{L}^q(\Omega)$ are compact for $q\in [1,\frac{dp}{d-p})$ if $p\in (1,d)$ and for arbitrary $q\in[1,\infty)$ if $p\geq d$ \cite[Remark 3.2]{rehbergsystems}. The proof is almost equal to the proofs of \cite[Part II,\, 5.6.1,\, Theorem 2]{evans} and \cite[Part II,\,5.7,\, Theorem 1]{evans}.
	\end{enumerate}
\end{remark}

\subsection{Operators and their properties}\label{sec:elliptic-regularity-for-systems-and-fractional-power-spaces}
In this subsection, we define the required Sobolev spaces of vector valued functions and introduce the operators $A_p$. Our notation differs from the one in \cite{rehbergsystems}. This is done in order to provide a structured framework for the construction of $A_p$ and to highlight the spaces on which each particular operator acts.
Results from the literature assure that $A_p$ satisfies the properties which we need for the analysis of \eqref{state_eqation_abstract} for particular values of $p$ to be chosen. 

We begin with two definitions \cite[Section 6]{rehbergsystems}:
\begin{definition}\label{Def:vector_Sobolev_space}
	With Assumption \ref{Ass:domain} and Assumption \ref{Ass:existence_extension} and $p\in [1,\infty)$ we define a Sobolev space of vector valued functions by the product space
	\begin{align*}
	\mathbb{W}_{\Gamma_D}^{1,p}(\Omega):= \prod\limits_{j=1}^m \mathrm{W}_{\Gamma_{D_j}}^{1,p}(\Omega).
	\end{align*}
	For $p\in (1,\infty)$ we denote its (componentwise) dual by $\mathbb{W}_{\Gamma_D}^{-1,p'}(\Omega)$. 
	
	We also define the operators
	\[\mathcal{L}_p:\mathbb{W}_{\Gamma_D}^{1,p}(\Omega)\rightarrow \mathrm{L}^p(\Omega, \mathbb{R}^{md}),\ \mathcal{L}_p(u):= \mathrm{vec}(\nabla u)=(\nabla u_1,\cdots, \nabla u_m)^\intercal\]

	and 
	\[I_p:\mathbb{W}_{\Gamma_D}^{1,p}(\Omega)\rightarrow \mathbb{W}_{\Gamma_D}^{-1,p}(\Omega),\ \langle I_p u,v \rangle_{\mathbb{W}_{\Gamma_D}^{1,p'}(\Omega)}:=\int_\Omega u\cdot v\,dx
	\ \forall v\in \mathbb{W}_{\Gamma_D}^{1,p'}(\Omega).\]
\end{definition}

Now we can define the operators $A_p$ and state the associate properties:
\begin{definition}\label{Def:A_p}
	Let the constants $d_1,\cdots, d_m > 0$ be given diffusion coefficients and
	\begin{align*}
	D=\mathrm{diag}(d_1,\cdots, d_1,\cdots, d_m,\cdots, d_m)\in \mathbb{R}^{md\times md}.
	\end{align*}
	
	For $p\in (1,\infty)$ we set
	\[
	\mathcal{A}_p: \mathbb{W}_{\Gamma_D}^{1,p}(\Omega) \rightarrow \mathbb{W}_{\Gamma_D}^{-1,p}(\Omega),\ \mathcal{A}_p:= \mathcal{L}_{p'}^* D \mathcal{L}_p.
	\]
	
	We define the unbounded operator
	\[
	A_p: \mathbb{W}_{\Gamma_D}^{-1,p}(\Omega) \rightarrow \mathbb{W}_{\Gamma_D}^{-1,p}(\Omega),\ A_p:=\mathcal{A}_p I_p^{-1}
	\]
	with domain
	\[
	\mathrm{dom}(A_p) = \mathrm{ran}\left(I_p\right) \subset \mathbb{W}_{\Gamma_D}^{-1,p}(\Omega),
	\]
	where $\mathrm{ran}\left(I_p\right)$ stands for the range of $I_p$.
	
%	For $\theta \geq 0$ the fractional power spaces $X^\theta:= \mathrm{dom}([A_p + 1]^\theta)\subset \mathbb{W}_{\Gamma_D}^{-1,p}(\Omega)$ and the unbounded operators $[A_p + 1]^\theta$ are well-defined \cite[Chapter 1]{henry}.
%	Note that $X^0=\mathbb{W}_{\Gamma_D}^{-1,p}(\Omega)$.
%	
%	$\mathrm{dom}(A_p)$ is equipped with the graph norm 
%\[\|y\|_{\mathrm{dom}(A_p)}=\|y\|_{\mathbb{W}_{\Gamma_D}^{-1,p}(\Omega)} + \|A_p y\|_{\mathbb{W}_{\Gamma_D}^{-1,p}(\Omega)}\]
%and $X^\theta$ with the norm
%\[
%\|y\|_{X^\theta}= \|(A_p+1)^\theta y\|_{\mathbb{W}_{\Gamma_D}^{-1,p}(\Omega)}.
%\]
\end{definition}

The following result is shown in \cite[Theorem 5.6 and Theorem 5.12]{rehbergsystems}:
\begin{theorem}\label{Thm:elliptic_regularity_for_systems}
	In the setting of Definition~\ref{Def:vector_Sobolev_space} and Definition~\ref{Def:A_p} there exists an open interval $\mathrm{J}$ around $2$ such that for all $p\in \mathrm{J}$ the operator $\mathcal{A}_p + I_p$ is a topological isomorphism between $\mathbb{W}_{\Gamma_D}^{1,p}(\Omega)$ and $\mathbb{W}_{\Gamma_D}^{-1,p}(\Omega)$.
	
	There is a constant $c>0$ such that for all $p\in \mathrm{J}$ and $\lambda\in \mathbb{C}_+:=\lbrace z\in \mathbb{C}: \mathrm{Re}z\geq 0 \rbrace$ the resolvent estimate 
	\begin{align*}
	\Vert (A_p+1+\lambda)^{-1} \Vert_{\mathcal{L}(\mathbb{W}_{\Gamma_D}^{-1,p}(\Omega))} \leq  \frac{c}{1+|\lambda|}
	\end{align*}
	holds true and $-A_p$ generates an analytic semigroup of operators on $\mathbb{W}_{\Gamma_D}^{-1,p}(\Omega)$.
\end{theorem}

\begin{remark}\label{Rem:frac_powers}
	Let $p\in \mathrm{J}$ with $\mathrm{J}$ from Theorem~\ref{Thm:elliptic_regularity_for_systems}.
	We equip $\mathrm{dom}(A_p)$ with the graph norm 
	\[\|y\|_{\mathrm{dom}(A_p)}=\|y\|_{\mathbb{W}_{\Gamma_D}^{-1,p}(\Omega)} + \|A_p y\|_{\mathbb{W}_{\Gamma_D}^{-1,p}(\Omega)}.\]
	Then $A_p$ is densely defined and closed and $\mathrm{dom}(A_p)$ is topologically equivalent to $\mathbb{W}_{\Gamma_D}^{1,p}(\Omega)$. Remark~\ref{Rem:embedding_w1p_lp}~(ii) therefore implies that $\mathrm{dom}(A_p)$ is compactly embedded into $\mathbb{W}_{\Gamma_D}^{-1,p}(\Omega)$.
	
	Furthermore, for $\theta \geq 0$ the fractional power spaces $X^\theta:= \mathrm{dom}([A_p + 1]^\theta)\subset \mathbb{W}_{\Gamma_D}^{-1,p}(\Omega)$ and the unbounded operators $[A_p + 1]^\theta$ are well-defined \cite[Chapter 1]{henry}.
	Note that $X^0=\mathbb{W}_{\Gamma_D}^{-1,p}(\Omega)$.
	In $X^\theta$ we use the norm
	\[
	\|y\|_{X^\theta}= \|(A_p+1)^\theta y\|_{\mathbb{W}_{\Gamma_D}^{-1,p}(\Omega)}.
	\]
	
	Also for $z\in \{\zeta \in \mathbb{C}:\ \mathrm{Re}(\zeta)> 0\}$ one can define the fractional powers $[A_p + 1]^z$ by the inverse of the operators $[A_p + 1]^{-z}$
%	, just as in Definition~\ref{Def:A_p}
	\cite[Chapter 7]{yagi2009abstract}. For $\theta \in \mathbb{R}$ and suitable $y\in \mathbb{W}_{\Gamma_D}^{-1,p}(\Omega)$ one can further define $[A_p + 1]^{i\theta}y$ by the limit of $[A_p + 1]^{z}y$ for $z\rightarrow i\theta$ with $\mathrm{Re}(z)> 0$. This leads to the notion of bounded purely imaginary powers of an operator \cite[Chapter 8]{yagi2009abstract}.
	We will not need the theory of purely imaginary powers in the rest of this paper. However, we will use the fact that $A_p+1$ has bounded purely imaginary powers for $p\in \mathrm{J}\cap [2,\infty)$ in order apply an existing result, which allows us to represent the spaces $X^\theta$ by complex interpolation spaces for $\theta\in (0,1)$, see Remark~\ref{Rem:maximal parabolic regularity} below.
\end{remark}

%The following result is shown in \cite[Theorem 5.6 and Theorem 5.12]{rehbergsystems}:
%\begin{theorem}\label{Thm:elliptic_regularity_for_systems}
%	In the setting of Definition~\ref{Def:vector_Sobolev_space} and Definition~\ref{Def:A_p} there exists an open interval $\mathrm{J}$ around $2$ such that for all $p\in \mathrm{J}$ the operator $\mathcal{A}_p + I_p$ is a topological isomorphism between $\mathbb{W}_{\Gamma_D}^{1,p}(\Omega)$ and $\mathbb{W}_{\Gamma_D}^{-1,p}(\Omega)$.
%	
%	There is a constant $c>0$ such that for all $p\in \mathrm{J}$ and $\lambda\in \mathbb{C}_+:=\lbrace z\in \mathbb{C}: \mathrm{Re}z\geq 0 \rbrace$ the resolvent estimate 
%	\begin{align*}
%	\Vert (A_p+1+\lambda)^{-1} \Vert_{\mathcal{L}(\mathbb{W}_{\Gamma_D}^{-1,p}(\Omega))} \leq  \frac{c}{1+|\lambda|}
%	\end{align*}
%	holds true and $-A_p$ generates an analytic semigroup of operators on $\mathbb{W}_{\Gamma_D}^{-1,p}(\Omega)$.
%\end{theorem}

We introduce the notion of maximal parabolic regularity \cite[Definition 2.7]{meyeroptimal} or \cite[Definition 11.2]{squareroot_problem}. This property allows us to improve the regularity of the mild solution $y$ of our evolution equation.

\begin{definition}\label{Def:maximal parabolic regularity}
For $p,q\in (1,\infty)$ and $(t_0,T)\subset \mathbb{R}$, we say that $A_p$ satisfies maximal parabolic $\mathrm{L}^{q}((t_0,T);\mathbb{W}_{\Gamma_D}^{-1,p}(\Omega))$-regularity if for all
$g\in \mathrm{L}^q\left((t_0,T);\mathbb{W}_{\Gamma_D}^{-1,p}(\Omega)\right)$ there is a unique solution $y\in \mathrm{W}^{1,q}((t_0,T);\mathbb{W}_{\Gamma_D}^{-1,p}(\Omega))\cap \mathrm{L}^{q}((t_0,T);\mathrm{dom}(A_p))$ of the equation
\[
\frac{d}{dt}y+A_p y = g,\ y(t_0)=0.
\]

The time derivative is taken in the sense of distributions \cite[Definition 11.2]{squareroot_problem}.

We abbreviate
\begin{align*}
&Y_q:= \mathrm{W}^{1,q}((0,T);\mathbb{W}_{\Gamma_D}^{-1,p}(\Omega))\cap \mathrm{L}^{q}((0,T);\mathrm{dom}(A_p)) \text{ and}\\
&Y_{q,t}:= \{y\in Y_q:\ y(t)=0\} \text{ for } t\in [0,T].
%&Y^*_{q,t}:= \{ y\in \mathrm{W}^{1,q}(0,T;[\mathrm{dom}(A_p)]^*)\cap \mathrm{L}^{q}((0,T); \mathbb{W}_{\Gamma_D}^{1,p'}(\Omega)):\ y(t)=0\}
\end{align*}
\end{definition}

\begin{remark}\label{Rem:maximal parabolic regularity}
	The following properties go along with maximal parabolic regularity:
	\begin{enumerate}
		\item
		Maximal parabolic regularity is independent of
		$q\in (1,\infty)$ and of the interval $(t_0,T)$ so that we just say that $A_p$ satisfies maximal parabolic regularity on $\mathbb{W}_{\Gamma_D}^{-1,p}(\Omega)$ \cite[Remark 11.3]{squareroot_problem}.
		\item
		If $A_p$ satisfies maximal parabolic regularity on $\mathbb{W}_{\Gamma_D}^{-1,p}(\Omega)$ then
		
		$(\frac{d}{dt}+A_p)^{-1}$
		is bounded as an operator from $\mathrm{L}^q((0,T);\mathbb{W}_{\Gamma_D}^{-1,p}(\Omega))$ to $Y_{q,0}$ \cite[Proof of Proposition 2.8]{meyeroptimal}.
		\item
		If $p\in \mathrm{J}\cap [2,\infty)$ with $\mathrm{J}$ from Theorem~\ref{Thm:elliptic_regularity_for_systems} then by \cite[Theorem 11.5]{squareroot_problem}, $A_p +1$ has bounded imaginary powers and satisfies maximal parabolic Sobolev regularity on $\mathbb{W}_{\Gamma_D}^{-1,p}(\Omega)$, see also Remark~\ref{Rem:frac_powers}.
		This yields that for $p\in \mathrm{J}\cap [2,\infty)$ also $A_p$ satisfies maximal parabolic Sobolev regularity on $\mathbb{W}_{\Gamma_D}^{-1,p}(\Omega)$ and with \cite[Theorem 11.6.1]{carracedo} we conclude that we have the topological equivalences	
		\begin{align*}
		[\mathbb{W}_{\Gamma_D}^{-1,p}(\Omega), \mathbb{W}_{\Gamma_D}^{1,p}(\Omega)]_{\theta} \simeq [\mathbb{W}_{\Gamma_D}^{-1,p}(\Omega), \mathrm{dom}(A_p)]_{\theta} \simeq X^\theta
		\end{align*}
		for $\theta\in (0,1)$.
		By $[\cdot, \cdot]_{\theta}$ we mean complex interpolation.
	\end{enumerate}

\end{remark}

The following embedding properties will be used several times \cite[Theorem 3]{amann_measures}: 
\begin{remark}\label{Rem:embeddings}
	Let $p\in \mathrm{J}$ with $\mathrm{J}$ from Theorem~\ref{Thm:elliptic_regularity_for_systems}. With $q\in (1,\infty)$ one has
	\begin{align*}
	Y_q &\hhookrightarrow \mathrm{C}^\beta((0,T); (\mathbb{W}_{\mathrm{\Gamma_D}}^{-1,p}(\Omega), \mathrm{dom}(A_p))_{\eta,1})
	\hookrightarrow \mathrm{C}^{\beta}((0,T); [\mathbb{W}_{\mathrm{\Gamma_D}}^{-1,p}(\Omega), \mathrm{dom}(A_p)]_{\theta})
	\text{ and }\\
	Y_q &\hhookrightarrow \mathrm{C}([0,T]; (\mathbb{W}_{\mathrm{\Gamma_D}}^{-1,p}(\Omega), \mathrm{dom}(A_p))_{\eta,q})
	\hookrightarrow \mathrm{C}([0,T]; [\mathbb{W}_{\mathrm{\Gamma_D}}^{-1,p}(\Omega), \mathrm{dom}(A_p)]_{\theta})
	\end{align*}
	for every $0<\theta < \eta < 1-1/q$ and $0\leq\beta < 1-1/q -\eta$. $(\cdot,\cdot)_{\eta,1}$ or $(\cdot,\cdot)_{\eta,q}$ respectively means real interpolation here. Compactness of the first embeddings follows because $\mathrm{dom}(A_p)$ is compactly embedded into $\mathbb{W}_{\mathrm{\Gamma_D}}^{-1,p}(\Omega)$, see Remark~\ref{Rem:frac_powers}.
	
%	Moreover it holds
%	\begin{align*}
%	Y_q &\hookrightarrow \mathrm{L}^{q/(1-qs)}((0,T); (\mathbb{W}_{\mathrm{\Gamma_D}}^{-1,p}(\Omega), \mathrm{dom}(A_p))_{\eta,1})\\
%	&\hookrightarrow \mathrm{L}^{q/(1-qs)}((0,T); [\mathbb{W}_{\mathrm{\Gamma_D}}^{-1,p}(\Omega), \mathrm{dom}(A_p)]_{\theta})
%	\end{align*}
%	for all $0<\theta < \eta <1-s$ and $0<s<1/p$ and also here the first embedding is compact.
\end{remark}

\begin{remark}
	For $p\in \mathrm{J}$ with $\mathrm{J}$ from Theorem~\ref{Thm:elliptic_regularity_for_systems}, we collect several estimates for the operator $(A_p+1)^\theta$ and the analytic semigroup $\exp(-A_p t)$:
	For $t>0$ and arbitrary $0<\gamma <1$ it is shown in \cite[Theorem 1.3.4]{henry} that for some $C>0$ one can estimate 
	\begin{align*}
	&\Vert \exp(-A_p t) \Vert_{\mathcal{L}(\mathbb{W}_{\mathrm{\Gamma_D}}^{-1,p}(\Omega))}\leq C \exp((1-\gamma) t)
	\text{ and }\\
	&\Vert (A_p+1)\exp(-A_p t) \Vert_{\mathcal{L}(\mathbb{W}_{\mathrm{\Gamma_D}}^{-1,p}(\Omega))}\leq \frac{C}{t} \exp((1-\gamma) t).
	\end{align*}
	
	Moreover for each $\theta\geq 0$, according to \cite[Theorem 1.4.3]{henry}, there is some 
	
	$C_\theta\in (0,\infty)$ such that
	\begin{align}\label{frac_pow_estimate1}
	\Vert (A_p+1)^{\theta}\exp(-A_p t) \Vert_{\mathcal{L}(\mathbb{W}_{\Gamma_D}^{-1,p}(\Omega))}\leq C_\theta t^{-\theta}\exp((1-\gamma) t).
	\end{align}
%	and for $0<\theta\leq 1$ and $x\in X^\theta$
%	\begin{align}\label{frac_pow_estimate2}
%	\Vert (\exp(-(A_p+1) t)-\mathrm{Id})x \Vert_{\mathcal{L}(\mathbb{W}_{\Gamma_D}^{-1,p}(\Omega))}\leq \frac{1}{\theta} C_{1-\theta} t^{\theta}\Vert (A_p+1)^{\theta}x \Vert_{\mathbb{W}_{\Gamma_D}^{-1,p}(\Omega)}.
%	\end{align}
	The constants $C_\theta$ are bounded if $\theta$ is contained in any compact subinterval of $(0,\infty)$ and also for $\theta\downarrow 0$.
\end{remark}

\subsection{Main assumption and notation}\label{Sec:Ass}
We collect several assumptions and introduce some short notation for the spaces and functions.

\begin{assumption}\label{Ass:general_ass_and_short_notation}
	We always suppose that Assumption~\ref{Ass:domain} and Assumption~\ref{Ass:existence_extension} hold.
	
	Moreover we assume:
	\begin{itemize}
		\item $d\geq 2$ and with $\mathrm{J}$ from Theorem~\ref{Thm:elliptic_regularity_for_systems} there holds $p\in \mathrm{J}\cap[2,\infty)$
%		 is such that $p$ and its Hölder conjugate $p'$ are contained in the interval
		and
		$2\geq p\left(1-\frac{1}{d}\right)$. 
		
		\item For some $w\in\mathbb{W}^{1,p'}_{\Gamma_D}(\Omega)\simeq [\mathbb{W}_{\Gamma_D}^{-1,p}(\Omega)]^*$
		%		\cap\mathrm{dom}([A_p^{1-\alpha}]^*)$
		the operator $S\in [\mathbb{W}_{\Gamma_D}^{-1,p}(\Omega)]^*$ from equation~\eqref{state_eqation_abstract} is given by
		\[
		S y = \langle y,w \rangle_{\mathbb{W}_{\Gamma_D}^{1,p'}(\Omega)} \ \forall y\in \mathbb{W}_{\Gamma_D}^{-1,p}(\Omega).
		\]
		%		\item $1>\alpha > \frac{1}{2} + \frac{1}{p}$.
		Note that $S$ belongs to $[X^\theta]^*$ for all $\theta\geq 0$ because of the embedding 
		
		$X^\theta\hookrightarrow \mathbb{W}_{\Gamma_D}^{-1,p}(\Omega)$.
		We assume $S\neq 0$.
		\item We will need a fractional power space with exponent strictly smaller than one. This fact is highlighted by a new parameter $\alpha$ instead of $\theta\in [0,\infty)$ from above.
		Assume that for some $\alpha\in (0,1)$ the function 
		$f:X^\alpha\times \mathbb{R}\rightarrow \mathbb{W}_{\Gamma_D}^{-1,p}(\Omega)$ is locally Lipschitz continuous with respect to the $X^\alpha$-norm.
		
		This means that for every $y_0\in X^\alpha$ there is a constant $L(y_0)$ and a neighbourhood 
		\[
		V(y_0)=\left\lbrace y\in X^\alpha: \Vert y-y_0 \Vert_{X^\alpha}\leq \delta \right\rbrace
		\] of $y_0$ such that
		\begin{align*}
		\Vert f(y_1,x_1) - f(y_2,x_2) \Vert_{X} \leq L(y_0) \left( \Vert y_1-y_2 \Vert_{\alpha} + \vert x_1-x_2 \vert \right)
		\end{align*}
		for every $y_1,y_2 \in V(y_0)$ and all $x_1,x_2\in \mathbb{R}$.
		
		%	$f$ is also directionally differentiable and therefore Hadamard directionally differentiable.
		
		Moreover, $f$ is assumed to have at most linear growth along solutions, i.e.
		\begin{align*}
		\Vert f(y,x) \Vert_{\mathbb{W}_{\Gamma_D}^{-1,p}(\Omega)} \leq M \left( 1+ \Vert y \Vert_{\alpha} + \vert x \vert \right)
		\end{align*}
		for some constant $M>0$.

	\end{itemize}
\end{assumption}

In the setting of Assumption~\ref{Ass:general_ass_and_short_notation} we collect the notation for the rest of the work:
\begin{itemize}
	\item For the particular $p$ from Assumption~\ref{Ass:general_ass_and_short_notation} we set
	\[X:=\mathbb{W}_{\Gamma_D}^{-1,p}(\Omega)\]
	with $\mathbb{W}_{\Gamma_D}^{-1,p}(\Omega)$ from Definition~\ref{Def:vector_Sobolev_space}.
	We sometimes identify elements $v\in X^*$ with their Riesz representation in $\mathbb{W}_{\Gamma_D}^{1,p'}(\Omega)$, i.e.
	\[
	\langle v,y \rangle_{X}= \langle y,v \rangle_{\mathbb{W}_{\Gamma_D}^{1,p'}(\Omega)}\ \forall y\in X.
	\]
	
	\item The operators $A_p$ and the spaces $X^\theta = \mathrm{dom}([A_p+1]^\theta)$ are defined as in Definition~\ref{Def:A_p} and Remark~\ref{Rem:frac_powers}.
	
	\item The spaces $Y_q$ and $Y_{q,t}$
	%		and $Y^*_{q,t}$
	are defined as in Definition~\ref{Def:maximal parabolic regularity}.
	
	\item $\mathcal{W}$ is the scalar stop operator. This operator is represented by \eqref{stop1}-\eqref{stop3}. Other representations can for example be found in \cite[Chapter III.3]{visintin2013differential}.
	
	\item We abbreviate $J_T=(0,T)$.
\end{itemize}

%\begin{assumption}\label{Ass:hadamard_hyst}
%	For arbitrary $T>0$, the stop operator $\mathcal{W}$ defined on the interval $[a,b]$ and for initial value $0$ is Lipschitz continuous on $\mathrm{C}(\overline{J_T})$.
%	
%	For all $q\in [1,\infty)$, the stop operator $\mathcal{W}:\mathrm{C}([t_0,T])\rightarrow \mathrm{L}^q(t_0,T)$ is Lipschitz continuous with some modulus $L_\mathcal{W}>0$ and directionally differentiable, and therefore Hadamard directionally differentiable  \cite[Proposition 5.5]{Brokate_weak_diff}.
%\end{assumption}

\subsection{Regularity of the stop operator}\label{SUBSEC:Stop}
The stop operator $\mathcal{W}$ which is represented by \eqref{stop1}-\eqref{stop3} is Lipschitz continuous as an operator on $\mathrm{C}(\overline{J_T})$ according to \cite[Part 1, Chapter III Lemma 2.1, Theorem 3.2 and Theorem 3.3]{visintin2013differential} with
\begin{align}
|\mathcal{W}[v_1](t)- \mathcal{W}[v_2](t)| \leq 2\sup\limits_{0\leq \tau \leq t} \vert v_1(\tau) -v_2(\tau) \vert && \text{ and } && \mathcal{W}[v](t) \leq 2\sup\limits_{0\leq \tau \leq t} \vert v(\tau) \vert + z_0\label{hyst_growth}
\end{align}
for all $v,v_1,v_2 \in \mathrm{C}(\overline{J_T})$ and $t\in [0,T]$. We have to add $z_0$ in \eqref{hyst_growth} because, by \eqref{stop3}, $\mathcal{W}[v](0)=z_0$ for any $v \in \mathrm{C}(\overline{J_T})$.

$\mathcal{W}$ is also bounded and weakly continuous on $\mathrm{W}^{1,q}(J_T)$ for $q\in [1,\infty)$ \cite[Part 1, Chapter III., Theorem 3.2]{visintin2013differential}.

\section{Well-posedness of the evolution equation}\label{Sec:Well_posed}
%corresponding to the system \eqref{state_equ_y}, \eqref{state_equ_z} 
We recap equation~\eqref{state_eqation_abstract} from the introduction which is
\begin{alignat*}{2}
\frac{d}{dt}y(t) + (A_p y)(t) &= (F[y])(t) + u(t)\ && \text{ in } X=\mathbb{W}_{\Gamma_D}^{-1,p}(\Omega) \text{ for } t>0,\\
y(0)&=0\in X,\notag
\end{alignat*}
where $(F[y])(t):=f(y(t),\mathcal{W}[S y](t))$. 
We recall that $X$ is a product of dual spaces.
In this section we show well-posedness of the problem. 
The first aim is to show that for every $u\in \mathrm{L}^q(J_T;X)$ with $q\in \left(\frac{1}{1-\alpha},\infty\right]$ problem \eqref{state_eqation_abstract} has a unique mild solution $y\in \mathrm{C}(\overline{J_T};X^\alpha)$, where $\alpha$ is fixed by Assumption~\ref{Ass:general_ass_and_short_notation}.
In particular, this means that $(F[y]) + u$ is contained in $\mathrm{L}^1(J_T;X)$ and that $y$ solves the integral equation
\begin{align}
y(t) = \int_{0}^{t} \exp(-A_p(t-s))[(F[y])(s) + u(s)] \, ds,\ t\in J_T\label{Def:mild_solution}
\end{align}
\cite[Definition 7.0.2]{lunardi}.
Afterwards we prove that the unique mild solution even belongs to $Y_{s,0}$ where $s=q$ if $q < \infty$ and with $s\in (1,\infty)$ arbitrary if $q = \infty$.

%The restriction of $\alpha$ to $(0,1/2)$ in Assumption~\ref{Ass:general_ass_and_short_notation} is only needed to assure $\frac{1}{1-\alpha} < 2$ so that the solution operator for problem \eqref{state_equ_y},\eqref{state_equ_z} is well defined.
%For the general statement it is necessary to assume $\alpha \in (0,1)$ and that the controls are contained in $\mathrm{L}^q(J_T;X)$ with $\frac{1}{1-\alpha}<q$.

\begin{theorem}\label{Thm:state_equ_wellposed}
	Let Assumption \ref{Ass:general_ass_and_short_notation} hold.
	
	Then for all $u\in \mathrm{L}^q(J_T;X)$ with $q\in \left(\frac{1}{1-\alpha},\infty\right]$ problem \eqref{state_eqation_abstract} has a unique mild solution
	
	$y=y(u)$ in $\mathrm{C}(\overline{J_T};X^\alpha)$. Note that $X^\alpha \subset X$ since $\alpha\in (0,1)$.

	The solution mapping 
	\[G:u \mapsto y(u),\ \mathrm{L}^q(J_T;X) \rightarrow \mathrm{C}(\overline{J_T};X^\alpha)\]
	is locally Lipschitz continuous.
	
	$G$ is linearly bounded with values in $\mathrm{C}(\overline{J_T};X^\alpha)$, i.e. for some $C=C(T)>0$ there holds
	\begin{align}
	\Vert G(u) \Vert_{\mathrm{C}(\overline{J_T};X^\alpha)} &\leq C(T) (1+\Vert u\Vert_{\mathrm{L}^q(J_T;X)})\label{solution_op_bdd}
	\end{align} 
	for all $u\in \mathrm{L}^q(J_T;X)$ and $C$ is independent of $u$.
	All statements remain valid if $\mathrm{C}(\overline{J_T};X^\alpha)$ is replaced by 
	$Y_{s,0}$ where $s=q$ if $q < \infty$ and $s\in (1,\infty)$ arbitrary if $q = \infty$.
\end{theorem}

\begin{proof}
	We prove the theorem as in \cite[Theorem 7.1.3]{lunardi} by a fixed point argument. Several estimates can be found in \cite[Appendix A]{meyeroptimal} in a similar form. We extend the results in \cite{lunardi} and \cite{meyeroptimal} by allowing for non-linearities which are only locally Lipschitz continuous and not Lipschitz continuous on bounded sets. Moreover, non-locality of the hysteresis operator in time requires additional work in several steps.
	We prove the theorem directly for $u\in \mathrm{L}^q(J_T;X)$ as it is done in \cite[Theorem 7.1.3]{lunardi}. In \cite[Appendix A]{meyeroptimal} the corresponding statement is first shown for smooth right hand sides and afterwards extended by a density argument.
	
	In the following, $c$ always denotes a generic constant which is adapted during the proof. 
	Note that for $\beta > -1$ there holds
	\begin{align}
		\int_{0}^{t} (t-s)^\beta \,ds = \frac{t^{1+\beta}}{1+\beta}.\label{integral_computaton}
	\end{align}
	The proof is divided into five steps.
	\begin{enumerate}[leftmargin=0.5cm]
		\item We show the existence of local solutions of problem \eqref{state_eqation_abstract}. 
	
	Consider $v_u(t):= \int\limits_{0}^t e^{-A_p(t-s)} u(s) \,ds$.	$v_u$ belongs to $\mathrm{C}(\overline{J_T};X^\alpha)$ for arbitrary $T>0$.
		
	Moreover, since $q'<\alpha^{-1}$, we have by \eqref{frac_pow_estimate1} and \eqref{integral_computaton}
	\begin{align}
	\Vert v_u \Vert_{\mathrm{C}(\overline{J_T};X^\alpha)} &\leq \left(\int_{0}^{T}\|e^{-A_p(t-s)}\|^{q'}_{\mathcal{L}(X,X^\alpha)} \,ds\right)^{1/q'} \Vert v_u \Vert_{\mathrm{L}^q(J_T;X)}\notag\\
	&\leq c e^{(1-\gamma)T} T^{1/q'-\alpha}\Vert v_u \Vert_{\mathrm{L}^q(J_T;X)} <\infty\label{convol_estim}.
	\end{align}
	
	Let $\delta> 0$ be small enough so that $f$ is Lipschitz continuous in $\overline{B_{X^\alpha}(0,\delta)}\times \mathbb{R}$ with a constant $L(0)>0$.
	
	We apply Assumption~\ref{Ass:general_ass_and_short_notation} and \eqref{hyst_growth} to estimate
	\begin{align}
	\Vert &(F[y_1])(t) - (F[y_2])(t) \Vert_{X}\notag\\
	&\leq L(0) \left(\Vert y_1(t) - y_2(t) \Vert_{X^\alpha} + 2\|S\|_{[X^\alpha]^*}\sup_{0\leq \tau \leq t} \Vert y_1(\tau) - y_2(\tau) \Vert_{X^\alpha}\right)\notag\\
	&\leq c \sup_{0\leq \tau \leq t} \Vert y_1(\tau) - y_2(\tau) \Vert_{X^\alpha}\label{Est:Thm:well_posed_1}
	\end{align}
%	and
%	\begin{align*}
%	\Vert f\left(y(t_1),\mathcal{W}[Sy](t_1)\right) - &f\left(y(t_2),\mathcal{W}[Sy](t_2)\right) \Vert_{X^\alpha}\\
%	&\leq L(0) (\Vert y(t_1) - v(t_2) \Vert_{X^\alpha} + \|S\|_{\mathcal{L}(X,X^\alpha)}\sup_{t_1\leq \tau \leq t_2} \Vert y(\tau) - v(\tau) \Vert_{X^\alpha})\label{Est:Thm:well_posed_2}
%	\end{align*}
	for all $y_1,y_2\in \overline{B_{\mathrm{C}(\overline{J_T};X^\alpha)}(0,\delta)}$ and $t\in \overline{J_T}$.
	
	The mapping 
	\begin{align*}
	\Phi_u(y)(t):= \int\limits_{0}^t e^{-A_p(t-s)} \left[ f\left(y(s),\mathcal{W}[Sy](s)\right) + u(s)\right] \,ds
	\end{align*}
	is well defined on $\mathrm{C}(\overline{J_T};X^\alpha)$. This is shown as in \cite[Appendix~A~(ii)]{meyeroptimal}.
	
	For $y_1,y_2\in \overline{B_{\mathrm{C}(\overline{J_T};X^\alpha)}(0,\delta)}$ we have by \eqref{frac_pow_estimate1}, \eqref{integral_computaton} and \eqref{Est:Thm:well_posed_1} that
	\begin{align*}
	\Vert \Phi_u(y_1) - \Phi_u(y_2) \Vert_{\mathrm{C}(\overline{J_T};X^\alpha)}
	&\leq \int_{0}^{T}\|e^{-A_p(t-s)}\|_{\mathcal{L}(X,X^\alpha)}\,ds \Vert F[y_1] - F[y_2] \Vert_{\mathrm{C}(\overline{J_T};X)} \\
	&\leq ce^{(1-\gamma)T}T^{1-\alpha}\Vert y_1-y_2 \Vert_{\mathrm{C}(\overline{J_T};X)} < \frac{1}{2}\Vert y_1-y_2 \Vert_{\mathrm{C}(\overline{J_T};X)}
	\end{align*} 
	for $T$ small enough.
	
	Consequently, in this case $\Phi_u$ is a $\frac{1}{2}$-contraction.
	
	Using this result together with \eqref{frac_pow_estimate1} and \eqref{integral_computaton} we obtain for $y\in\overline{B_{\mathrm{C}(\overline{J_T};X^\alpha)}(0,\delta)}$
	\begin{align*}
	\Vert \Phi_u(y)(t)\Vert_{X^\alpha} &\leq \Vert \Phi_u(y)(t) - \Phi_u(0)(t) \Vert_{X^\alpha} + \Vert \Phi_u(0)(t)\Vert_{X^\alpha}\\
	&\leq \frac{\delta}{2} +  \left(
	\int_{0}^{T}\|e^{-A_p(t-s)}\|^{q'}_{\mathcal{L}(X,X^\alpha)}\,ds
	\right)^{1/q'}\Vert F[0]  + u\Vert_{\mathrm{L}^q(J_T;X)}\\
%	&\leq \frac{\delta}{2} +  ce^{(1-\gamma)T}\left(\int_{0}^{T}(t-s)^{-q'/\alpha}\,ds\right)^{1/q'}\Vert f\left(0,0\right) + u\Vert_{\mathrm{L}^q(J_T;X)}\\
	&\leq \frac{\delta}{2} +  ce^{(1-\gamma)T}T^{1/q'-\alpha}\Vert f\left(0,z_0\right) + u\Vert_{\mathrm{L}^q(J_T;X)}	\leq \delta
	\end{align*}
	if $T$ is small enough.
	
	Because $\Phi_u$ then maps $\overline{B_{\mathrm{C}(\overline{J_T};X^\alpha)}(0,\delta)}$ into itself and since $\overline{B_{\mathrm{C}(\overline{J_T};X^\alpha)}(0,\delta)}$ is a closed subset of $\mathrm{C}(\overline{J_T};X^\alpha)$, Banach's fixed point theorem yields a unique fixed point $y$ of $\Phi_u$ in $\overline{B_{\mathrm{C}(\overline{J_T};X^\alpha)}(0,\delta)}$.
	
	This fixed point defines a (local) mild solution of problem \eqref{state_eqation_abstract} in $\overline{J_T}$ \cite[Definition 7.0.2]{lunardi}.
	
	\item We show that global mild solutions for problem \eqref{state_eqation_abstract} exist and boundedness of the solution mapping $G$.
	
	This part requires some cautiousness because the hysteresis operator is non-local in time.
	
	Remember that the local mild solution $y$ of \eqref{state_eqation_abstract} takes the form \eqref{Def:mild_solution}.
	With Assumption~\ref{Ass:general_ass_and_short_notation} and the second estimate in \eqref{hyst_growth} we estimate
	
	$|\mathcal{W}[Sy](t)| \leq 2\|S\|_{[X^\alpha]^*} \sup_{0\leq \tau \leq t}\|y(\tau)\|_{X^\alpha} + |z_0|$ for all $t\in \overline{J_T}$.
	
	Moreover, by \eqref{frac_pow_estimate1} there holds
	\begin{align*}
	\Vert (A_p+1)^{\alpha}\exp(-A_p t) \Vert_{\mathcal{L}(\mathbb{W}_{\Gamma_D}^{-1,p}(\Omega))}\leq C_\alpha t^{-\alpha}\exp((1-\gamma) t).
	\end{align*}
	Equation~\eqref{integral_computaton} yields
	\begin{align*}
	\left(\int_{0}^{t} (t-s)^{-\alpha q'} \,ds\right)^{1/q'} = \left(\frac{t^{1-\alpha q'}}{1-\alpha q'}\right)^{1/q'} = \frac{t^{1/q'-\alpha}}{(1-\alpha q')^{1/q'-\alpha}}.
	\end{align*}
	We combine those three observations to obtain a bound for the norm of $y(t)$ for all $t\in \overline{J_T}$ in the form
	\begin{align}
	\Vert y(t)\Vert_{X^\alpha} &\leq ce^{(1-\gamma)T}\left[\int_{0}^{t}(t-s)^{-\alpha} \left(1+3\sup_{0\leq \tau \leq s}\|y(\tau)\|_{X^\alpha} + |z_0|\right)\,ds + t^{1/q'-\alpha}\Vert u\Vert_{\mathrm{L}^q(J_T;X)}\right]\notag\\
	&\leq c_0(T)\int_{0}^{t}(t-s)^{-\alpha} \sup_{0\leq \tau \leq s}\|y(\tau)\|_{X^\alpha}\,ds + c_1(T) [1+ \Vert u\Vert_{\mathrm{L}^q(J_T;X)}],\label{key} 
	\end{align}
	where $c_0(T),c_1(T)>0$ are constants which depend on $T$ (and on $q'$ and the fixed value $\alpha$).
	As in the proof of \cite[Theorem 6.3.3]{pazy} note that if the solution of \eqref{state_eqation_abstract} exists on $[0,T[$ it can be continued as long as $\Vert y(t) \Vert_{X^\alpha}$ remains bounded with $t\uparrow T$.
	
	Clearly this is the case if 
	\begin{align}\label{bdd1}
	\sup\limits_{0\leq \tau < T} \Vert y(\tau) \Vert_{X^\alpha} \leq C(T)
	\end{align}
	for some $C(T)>0$.
	
	It is not hard to show that the function $t\mapsto \sup\limits_{0\leq \tau < t} \Vert y(\tau) \Vert_{X^\alpha}$ is continuous on $[0,T[$.

	We prove that for $t\in J_T$ the function
	\[
	g:\tau \mapsto \int\limits_{0}^{\tau} (\tau -s )^{-\alpha} \sup\limits_{0\leq \tau' \leq s} \Vert y(\tau') \Vert_{X^\alpha}  \,ds ,\ \tau \in \overline{J_t}\] 
	is monotone increasing.
	
	Let $t_0\in \overline{J_t}$ and $\delta> 0$ be given. Then by a shift of the integration interval we obtain
	\begin{align*}
	&g(t_0 + \delta)-g(t_0) \\
	&= \int\limits_{0}^{t_0 + \delta} (t_0 + \delta -s )^{-\alpha} \sup\limits_{0\leq \tau' \leq s} \Vert y(\tau') \Vert_{X^\alpha}  \,ds 
	- \int\limits_{0}^{t_0} (t_0 -s )^{-\alpha}  \sup\limits_{0\leq \tau' \leq s} \Vert y(\tau') \Vert_{X^\alpha}\\
%	&= \int\limits_{\delta}^{t_0 + \delta} (t_0 + \delta -s )^{-\alpha} \sup\limits_{0\leq \tau' \leq s} \Vert y(\tau') \Vert_{X^\alpha}  \,ds 
%	+\int\limits_{0}^{\delta} (t_0 + \delta -s )^{-\alpha} \sup\limits_{0\leq \tau' \leq s} \Vert y(\tau') \Vert_{X^\alpha}  \,ds\\
%	&- \int\limits_{0}^{t_0} (t_0 -s )^{-\alpha}  \sup\limits_{0\leq \tau' \leq s} \Vert y(\tau') \Vert_{X^\alpha}\\
	&=  \int\limits_{0}^{t_0} (t_0 -s )^{-\alpha} \left(\sup\limits_{0\leq \tau' \leq s + \delta} \Vert y(\tau') \Vert_{X^\alpha}  
	- \sup\limits_{0\leq \tau' \leq s} \Vert y(\tau') \Vert_{X^\alpha}\right) \,ds \\
	&+\int\limits_{0}^{\delta} (t_0 + \delta -s )^{-\alpha} \sup\limits_{0\leq \tau' \leq s} \Vert y(\tau') \Vert_{X^\alpha}  \,ds \geq 0.
	\end{align*}
	
	Because $g$ is monotone increasing we can take the supremum in \eqref{key} on both sides to get
	\begin{align*}
	&\sup\limits_{0\leq \tau \leq t} \Vert y(\tau) \Vert_{X^\alpha}\leq c_0(T)\int_{0}^{t}(t-s)^{-\alpha} \sup_{0\leq \tau \leq s}\|y(\tau)\|_{X^\alpha}\,ds + c_1(T) [1+ \Vert u\Vert_{\mathrm{L}^q(J_T;X)}].
	\end{align*}
	By Gronwall's Lemma this implies
	\begin{align*}
	\sup\limits_{0\leq \tau \leq t} \Vert y(\tau) \Vert_{X^\alpha} &\leq C(T) (1+\Vert u\Vert_{\mathrm{L}^q(J_T;X)})
	\end{align*}
	for $C(T)>0$ and for all $t\in \overline{J_T}$ \cite[Lemma 6.7]{pazy}, which proves \eqref{bdd1}.

	\item Local Lipschitz continuity of the solution mapping is shown in a similar way as global existence but we have to be careful because $f$ is only locally Lipschitz continuous and not Lipschitz continuous on bounded sets as is the case in \cite{meyeroptimal}.
	
%	For any $v\in \mathbb{R}$ t
	The function $(y(\cdot),v)\mapsto f(y(\cdot),v)$ is locally Lipschitz continuous from $\mathrm{C}(\overline{J_T};X^\alpha)\times \mathbb{R}$ to $\mathrm{C}(\overline{J_T};X)$ with respect to the $\mathrm{C}(\overline{J_T};X^\alpha)$-norm. 
	To see this, note first that the set $y(\overline{J_T})\subset X^\alpha$ is compact for given $y\in \mathrm{C}(\overline{J_T};X^\alpha)$ by continuity of $y$ and since the interval $\overline{J_T}\subset\mathbb{R}$ is a compact set. Moreover, $y(\overline{J_T})$ equipped with the subspace topology in $X^\alpha$ is separable, again because $y$ is continuous and since $\overline{J_T}$ is separable. Let $\{x_i\}_{i\in \mathbb{N}}\subset X^\alpha\cap y(\overline{J_T})$ be a dense subset of $y(\overline{J_T})$. The function $(\tilde{y},v)\mapsto f(\tilde{y},v)$ is locally Lipschitz continuous from $X^\alpha\times\mathbb{R}$ to $X$. So one can find constants $\varepsilon(x_i)>0$ such that $(\tilde{y},v)\mapsto f(\tilde{y},v)$ is Lipschitz continuous on $B_{X^\alpha}(x_i,\varepsilon(x_i))\times\mathbb{R}$. Because $\{x_i\}_{i\in \mathbb{N}}$ is dense in $y(\overline{J_T})$, it follows that the set $y(\overline{J_T})$ is contained in $\cup_{i\in I} B_{X^\alpha}(x_i,\varepsilon(x_i))$. Since $y(\overline{J_T})$ is compact in $X^\alpha$, one can find a finite subcover $\cup_{i=1}^k B_{X^\alpha}(x_i,\varepsilon(x_i))$ which still contains $y(\overline{J_T})$. Now the function $(\tilde{y},v)\mapsto f(\tilde{y},v)$ is Lipschitz continuous on $\cup_{i=1}^k B_{X^\alpha}(x_i,\varepsilon(x_i))\times \mathbb{R}$ with a modulus given by the maximum of the Lipschitz constants on $B_{X^\alpha}(x_i,\varepsilon(x_i))\times \mathbb{R}$ over all $i\in \{1,\cdots k\}$. Since $V_y:=\{ \tilde{y}\in \mathrm{C}(\overline{J_T};X^\alpha):\ y(t)\in B_{X^\alpha}(x_i,\varepsilon(x_i))\ \forall t\in \overline{J_T}\}$ is a neighbourhood of $y$ in $\mathrm{C}(\overline{J_T};X^\alpha)$, this proves that  $(\tilde{y}(\cdot),v)\mapsto f(\tilde{y}(\cdot),v)$ is Lipschitz continuous from $B_{X^\alpha}(x_i,\varepsilon(x_i))\times\mathbb{R}\subset \mathrm{C}(\overline{J_T};X^\alpha)\times\mathbb{R}$ to $\mathrm{C}(\overline{J_T};X)$, i.e. $(y(\cdot),v)\mapsto f(y(\cdot),v)$ is locally Lipschitz continuous from $\mathrm{C}(\overline{J_T};X^\alpha)\times \mathbb{R}$ to $\mathrm{C}(\overline{J_T};X)$ with respect to the $\mathrm{C}(\overline{J_T};X^\alpha)$-norm. Moreover, 
	there even holds the pointwise estimate \begin{align}\label{Eq:Lipschitz_continuity_f}
	\|f(y_1(t),v_1) - f(y_2(t),v_2)\|_X\leq L(y) (\|y_1(\tau) - y_2(\tau)\|_{X^\alpha} + |v_1-v_2|)
	\end{align} 
	for all $y_1,y_2\in V_y$, $v_1,v_2\in \mathbb{R}$ and $t\in \overline{J_T}$ and for some $L(y)>0$.
	
	Lipschitz continuity of $\mathcal{W}$, see Subsection~\ref{SUBSEC:Stop}, together with Assumption~\ref{Ass:general_ass_and_short_notation} yields that also $y\mapsto F[y]$ is locally Lipschitz continuous from $\mathrm{C}(\overline{J_T};X^\alpha)$ to $\mathrm{C}(\overline{J_T};X)$ and for $y\in \mathrm{C}(\overline{J_T};X^\alpha)$ there exists a neighbourhood $V_y$ of $y$ and a constant $L(y)>0$ such that the pointwise estimate
	\begin{align}\label{Eq:Lipschitz_continuity_F}
		\|F(y_1)(t) - F(y_2)(t)\|_X\leq L(y)\sup_{0\leq \tau \leq t} \|y_1(\tau) - y_2(\tau)\|_{X^\alpha}
	\end{align}
	holds
	for all $y_1,y_2\in V$ and $t\in \overline{J_T}$.
	Let $y=G(u)$ be the solution of problem \eqref{state_eqation_abstract} corresponding to $u$.
%	 and $\tilde{u}$ where $\tilde{u}\in \overline{B_{\mathrm{L}^q(J_T;X)}(u,R)}$ is arbitrary for some fixed $R>0$. 
	Moreover, let $\delta>0$ be small enough so that $F$ is Lipschitz continuous in $\overline{B_{\mathrm{C}(\overline{J_T};X^\alpha)}(y,\delta)}$ with modulus $L(y)$.
	
%	Let $t_{\max}\in (0,T]$ be defined as
%	\[
%	t_{\max}= \sup_{\tau \in (0,T]} 
%	\left\{
%		\exists R>0 \text{ s.t. } \sup_{0\leq t < \tau} \Vert y(t) - G(\tilde{u})(t) \Vert_{X^\alpha} <\delta\ \forall \tilde{u}\in \overline{B_{\mathrm{L}^q(J_T;X)}(u,R)}
%	\right\}.
%	\]
	For $R>0$ to be chosen let $\tilde{u}\in \overline{B_{\mathrm{L}^q(J_T;X)}(u,R)}$ be arbitrary. There holds $y(0)=G(\tilde{u})(0)=0$.
	Continuity of $y$ and $G(\tilde{u})$ yields that we can find some $\tau>0$ such that
	\[
	\sup_{0\leq t < \tau} \Vert y(t) - G(\tilde{u})(t) \Vert_{X^\alpha} <\delta.
	\]
%	Assume that $t_{\max}<T$.
	
	With \eqref{frac_pow_estimate1}, \eqref{hyst_growth}, \eqref{integral_computaton}, \eqref{Eq:Lipschitz_continuity_F} and Assumption~\ref{Ass:general_ass_and_short_notation} we obtain
	\begin{align*}
	&\Vert y(t) - G(\tilde{u})(t) \Vert_{X^\alpha} \\
	&\leq ce^{(1-\gamma)T} \int\limits_{0}^{t} (t-s)^{-\alpha} [\Vert (F[y])(s)- (F[G(\tilde{u})])(s) \Vert_{X} + \|u-\tilde{u}\|_X]\,ds\\
	&\leq c(T,y) \int\limits_{0}^{t} (t-s)^{-\alpha} \sup\limits_{0\leq \tau \leq s} \Vert y(\tau) - G(\tilde{u})(\tau) \Vert_{\alpha} \,ds + c\Vert u-\tilde{u} \Vert_{\mathrm{L}^q(J_T;X)}
	\end{align*}
	for $t\in[0,\tau)$ and constants $c(T,y),c>0$.
	
	Similar as in Step 2 one can use Gronwall's Lemma to prove that there is some $C(T,y)>0$ such that
	\[
	\sup_{0\leq t \leq \tau} \Vert y(t) - G(\tilde{u})(t) \Vert_{X^\alpha} \leq C(T,y) \Vert u-\tilde{u}\Vert_{\mathrm{L}^q(J_T;X)} < \delta
	\]
	if $R$ is chosen small enough, since $\tilde{u}\in \overline{B_{\mathrm{L}^q(J_T;X)}(u,R)}$.
	Repeating the argument shows that 
	\[
	\sup_{0\leq t \leq T} \Vert y(t) - G(\tilde{u})(t) \Vert_{X^\alpha} \leq C(T,y) \Vert u-\tilde{u}\Vert_{\mathrm{L}^q(J_T;X)} < \delta
	\]
	for some appropriate $R>0$ and all $\tilde{u}\in \overline{B_{\mathrm{L}^q(J_T;X)}(u,R)}$.
	This implies that $G$ maps $\overline{B_{\mathrm{L}^q(J_T;X)}(u,R)}$ into $\overline{B_{\mathrm{C}(\overline{J_T};X^\alpha)}(y,\delta)}$ and $F$ is Lipschitz continuous on this set.
%	This is a contradiction to the definition of $t_{\max}$.
%	Consequently, $t_{\max}=T$ so that we can choose $R$ small and obtain by 
	A similar computation yields a constant $C(T,y)>0$ such that for arbitrary $u_1,u_2\in \overline{B_{\mathrm{L}^q(J_T;X)}(u,R)}$ there holds 
	\[
	\sup_{0\leq t \leq T} \Vert G(u_1)(t) - G(u_2)(t) \Vert_{X^\alpha} \leq C(T,y) \Vert u_1-u_2\Vert_{\mathrm{L}^q(J_T;X)}.
	\]
	This proves that $G$ is Lipschitz continuous in $\overline{B_{\mathrm{L}^q(J_T;X)}(u,R)}$. 
	
	So we have shown local Lipschitz continuity of $G$ from $\mathrm{L}^q(J_T;X)$ to $\mathrm{C}(\overline{J_T};X^\alpha)$.
	
	\item Uniqueness of the mild solution follows by local Lipschitz continuity of $G$ if one inserts $u_1=u_2$.
	
	\item The last statement of the theorem follows from maximal parabolic Sobolev regularity of $A_p$, cf. Remark~\ref{Rem:maximal parabolic regularity}. One applies $(\frac{d}{dt}+A_p)^{-1}$ to $F[y] + u \in \mathrm{L}^q(J_T;X)$ or to $F[y_1]-F[y_2] + u_1 - u_2 \in \mathrm{L}^q(J_T;X)$ respectively, see also \cite[Proposition 2.8]{meyeroptimal}. Note that this is the only step of the proof in which we must assume $p\in \mathrm{J}\cap [2,\infty)$ with $\mathrm{J}$ from Theorem~\ref{Thm:elliptic_regularity_for_systems}. In the previous steps also $p\in \mathrm{J}$  would have been sufficient.
\end{enumerate}
\end{proof}

\section{Hadamard directional differentiability}\label{Sec:Hadam_diff}
\subsection{Definition and properties}

We want to show differentiability of the solution mapping $G$ for problem \eqref{state_eqation_abstract}. Because the hysteresis operator is not smooth we can not expect a Fréchet derivative.
Therefore we consider a weaker form of differentiability, the Hadamard directional derivative \cite{bonnans, Brokate_weak_diff}.

To start with, we define directional differentiability of a mapping
$g:U\subset X\rightarrow Y$ from an open set $U\subset X$ of a normed vector space $X$ into a normed vector space $Y$ \cite[Definition 2.44]{bonnans}:
\begin{definition}
	Let $X,Y$ be normed vector spaces.
	We call $g$ directionally differentiable at $x\in U\subset X$ in the direction $h\in X$ if
	\begin{align*}
	g'[x;h]:= \lim\limits_{\lambda \downarrow 0} \frac{g(x+\lambda h) - g(x)}{\lambda}
	\end{align*}
	exits in $Y$. If $g$ is directionally differentiable at $x$ in every direction $h$ we call $g$ directionally differentiable at $x$.
\end{definition}

Using this definition we introduce the concept of the Hadamard directional derivative:
\begin{definition}
	If $g$ is directionally differentiable at $x\in U$ and if in addition for all functions $r:[0,\lambda_0)\rightarrow X$ with $\lim\limits_{\lambda \rightarrow 0} \frac{r(\lambda)}{\lambda}=0$
	\begin{align*}
	g'[x;h]= \lim\limits_{\lambda \downarrow 0} \frac{g(x+\lambda h + r(\lambda)) - g(x)}{\lambda}
	\end{align*}
	for all directions $h\in X$, we call $g'[x;h]$ the Hadamard directional derivative of $g$ at $x$ in the direction $h$. 
	
	Note that $g(x+\lambda h + r(\lambda))$ is only well defined if $\lambda$ is already small enough so that $x+\lambda h + r(\lambda) \in U$.
\end{definition}
We will frequently use the following properties of the concept of Hadamard directional differentiability:
\begin{lemma}\label{Lem:hadam_chain_rule}
	\cite[Proposition 2.47]{bonnans}
	Suppose that $g:U\subset X \rightarrow Y$ is Hadamard directionally differentiable at $x\in U$ and that $f:V\subset g(U) \rightarrow Z$ is Hadamard directionally differentiable at $g(x)\in V$.
	Then $f \circ g : U \rightarrow Z$ is Hadamard directionally differentiable at $x$ and
	\begin{align*}
	(f \circ g)'[x;h] = f'\left[g(x);g'[x;h]\right].
	\end{align*}
\end{lemma}

\begin{lemma}\label{Lem:hadam_for_lip}
	\cite[Proposition 2.49]{bonnans}
	Suppose that $g:U\subset X \rightarrow Y$ is directionally differentiable at $x\in U$ and in addition Lipschitz continuous with modulus $c(x)$ in a neighbourhood of $x$. Then $g$  is Hadamard directionally differentiable at $x$ and $g'[x;\cdot]$ is Lipschitz continuous on $X$ with modulus $c(x)$.
\end{lemma}

\subsection{Hadamard differentiability of the stop operator}\label{SUBSEC:Hadam_Stop}
The stop operator $\mathcal{W}$ from Subsection~\ref{SUBSEC:Stop} is Hadamard directionally differentiable as a mapping $\mathrm{C}[0,T]\rightarrow \mathrm{L}^q(0,T)$.
This follows from the corresponding result for the play operator \cite[Proposition 5.5]{Brokate_weak_diff} and because 
\[
\mathcal{P}+\mathcal{W} = \mathrm{Id}
\]
defines a scalar play operator $\mathcal{P}$
\cite[Part 1 Chapter III Proposition 3.3]{visintin2013differential}.

In \cite{Brokate_weak_diff} Hadamard directional differentiability is only proved for the case when

$[a,b]=[-r,r]$ for some $r>0$ and for the corresponding symmetrical play $\mathcal{P}_r$.

One can generalize this result by concatenating $\mathcal{P}_r$ with $r=\frac{b-a}{2}$ with the affine linear transformation
\[
\mathcal{T}:[-r,r]\rightarrow [a,b],\ \mathcal{T}:x\mapsto x+\frac{b+a}{2}.
\]
Then for piecewise monotone input functions $v$ with a monotonicity partition

$0=t_0\leq \cdots \leq t_k=T$ and for $z(t_i):=\mathcal{P}_r(\mathcal{T}(v))(t_i)$, $z(0)=z_0$ we obtain inductively
\begin{align*}
\mathcal{P}_r(\mathcal{T}(v))(t) &= \max \lbrace \mathcal{T}(v(t)) - r  , \min \lbrace \mathcal{T}(v(t)) + r, z(t_i) \rbrace\rbrace\\
& = \max \lbrace v(t) +a  , \min \lbrace v(t) + b, v(t_i) \rbrace\rbrace = \mathcal{P}(v)(t)
\end{align*}
for $t\in [t_i,t_i+1]$. By extension to more general input functions it follows $\mathcal{P}_r \circ \mathcal{T} = \mathcal{P}$.

Differentiability of $\mathcal{T}$ and Hadamard directional differentiability of $\mathcal{P}_r$ together with the chain rule yield Hadamard directional differentiability for $\mathcal{P}$ and then also for $\mathcal{W}$.

\begin{remark}\label{Rem:generalize_hyst_op}
	As already mentioned in the introduction, all our results hold if we replace the stop operator by $\mathcal{P}$ or by another hysteresis operator with appropriate properties. The main reason why we decided for $\mathcal{W}$ is the following: In Section~\ref{Sec:Optimal_contr}, we apply our results to an optimal control problem in which \eqref{state_eqation_abstract} is the state equation. We will derive an adjoint system for this problem in a forthcoming paper. This is achieved by a regularization of \eqref{stop1}-\eqref{stop3}. 
\end{remark}

\subsection{Hadamard differentiability of the solution operator for the evolution equation}

We want to prove Hadamard directional differentiability of the solution operator for problem \eqref{state_eqation_abstract}.

\begin{assumption}\label{Ass:General_for_Hadam}
	In addition to Assumption~\ref{Ass:general_ass_and_short_notation} we assume that $f$ is directionally differentiable and therefore Hadamard directionally differentiable.
\end{assumption}

The statement of the following theorem is almost equal to \cite[Theorem 3.2]{meyeroptimal}, but the proof is different due to the hysteresis operator and because our function $f$ is only locally Lipschitz continuous.

Also, in Step 2 of the following proof we show a statement which is very similar to \cite[Lemma 3.1]{meyeroptimal}, but again the proof has to be different in our setting.

\begin{theorem}\label{Thm:Solution_op_Hadamard}
	Let Assumption~\ref{Ass:General_for_Hadam} hold.
	
	For any $q\in \left(\frac{1}{1-\alpha},\infty\right)$ the solution operator $G:\mathrm{L}^q(J_T;X)\rightarrow \mathrm{C}(\overline{J_T};X^\alpha)$ of problem \eqref{state_eqation_abstract} is Hadamard directionally differentiable.
	
	Its derivative $y^{u,h}:=G'[u;h]$ at $u\in \mathrm{L}^q(J_T;X)$ in direction $h\in \mathrm{L}^q(J_T;X)$ is given by the unique mild solution $\zeta\in\mathrm{C}(\overline{J_T};X^\alpha)$ of
	\begin{alignat*}{2}
		\dot{\zeta}(t) + (A_p \zeta)(t)&= F'[y;\zeta](t) + h(t)&&\ \text{in } J_T,\\
		\zeta(0)&=0,
	\end{alignat*}
	where $F'[y;\zeta](t)=f'[(y(t),\mathcal{W}[Sy](t));(y(t),\mathcal{W}'[Sy;S\zeta](t))]$ and $y=G(u)$, see Theorem~\ref{Thm:state_equ_wellposed}.
	
	Moreover, $G'[u;h]\in Y_{q,0}$ is the Hadamard directional derivative of $G:\mathrm{L}^q(J_T;X)\rightarrow Y_{q,0}$.
	
	The mapping $h\mapsto G'[u;h]$ is Lipschitz continuous from $\mathrm{L}^q(J_T;X)$ to $\mathrm{C}(\overline{J_T};X^\alpha)$ and to $Y_{q,0}$ with a modulus of continuity $c=C(G(u),T)$.
\end{theorem}

\begin{proof}
	We show the theorem in five steps.
	\begin{enumerate}[leftmargin=0.5cm]
		\item First we prove that the function $\tilde{F}: \mathrm{C}(\overline{J_T};X^\alpha)\times \mathrm{L}^q(J_T) \rightarrow \mathrm{L}^q(J_T;X)$,
		\[
		\tilde{F}:(y,v) \mapsto \left[ t\mapsto f(y(t),v(t)) \right]
		\] is Hadamard directionally differentiable.
	We want to use Lemma \ref{Lem:hadam_for_lip}.
		
		\begin{enumerate}[leftmargin=0.1cm]
			\item We show that $\tilde{F}$ is well-defined. 
			
			Since $q> 1$ we have for $x_1,x_2 \in \mathbb{R}_+$
			\[
			(x_1 + x_2)^q \leq 2^{q-1} (x_1^q + x_2^q).
			\]
			Let $(y,v) \in \mathrm{C}(\overline{J_T};X^\alpha)\times \mathrm{L}^q(J_T)$ be given.
			
			Measurability of $\tilde{F}(y,v)$ follows from measurability of $y$ and $v$ and from continuity of $f$ in both components.

			Furthermore, for a.e. $s\in J_T$ we estimate
			\begin{align*}
			\Vert f(y(s),v(s)) \Vert_{X}^q 
			&\leq M^q  (\Vert y(s)\Vert_{X^\alpha} + \vert v(s) \vert + 1)^q
			\leq M^q 2^{q-1}  (\Vert y(s) \Vert_{X^\alpha} +1)^q +  \vert v(s) \vert^q
			\end{align*}
			with $M$ from Assumption~\ref{Ass:general_ass_and_short_notation}, so that $\tilde{F}(y,v)\in \mathrm{L}^q(J_T;X)$.
			
			\item 
			We show that $\tilde{F}$ is locally Lipschitz continuous with respect to the $\mathrm{C}(\overline{J_T};X^\alpha)$-norm.
			As in Step 3 in the proof of Theorem~\ref{Thm:state_equ_wellposed} note that $(y(\cdot),v)\mapsto f(y(\cdot),v)$ is locally Lipschitz continuous from $\mathrm{C}(\overline{J_T};X^\alpha)\times \mathbb{R}$ to $\mathrm{C}(\overline{J_T};X)$ with respect to the $\mathrm{C}(\overline{J_T};X^\alpha)$-norm.
			
			For $y\in \mathrm{C}(\overline{J_T};X^\alpha)$ let $\overline{B_{\mathrm{C}(\overline{J_T};X^\alpha)}(y,\delta)}\times \mathbb{R}$ be given such that this function is Lipschitz continuous with modulus $L(y)$.
			
			Consider any $y_1,y_2\in \overline{B_{\mathrm{C}(\overline{J_T};X^\alpha)}(y,\delta)}$ and $v_1,v_2\in \mathrm{L}^q(J_T)$.
			
			By \eqref{Eq:Lipschitz_continuity_f} we obtain for a.e. $s\in J_T$
			\begin{align*}
			&\Vert \tilde{F}(y_1,v_1)(s) - \tilde{F}(y_2,v_2)(s) \Vert_{X} 
			\leq  L(y)\left[  \Vert y_1(s) - y_2(s) \Vert_{X^\alpha} + \vert v_1(s) - v_2(s) \vert \right].
			\end{align*}
			
			Minkowski's inequality and $
			\Vert y_1 - y_2 \Vert_{\mathrm{L}^q(J_T;X^\alpha)}
			\leq T^{1/q}\Vert y_1 - y_2 \Vert_{\mathrm{C}(\overline{J_T};X^\alpha)}$ yields 
			\begin{align}
			&\Vert \tilde{F}(y_1,v_1) - \tilde{F}(y_2,v_2) \Vert_{\mathrm{L}^q(J_T;X)}
			\leq L(y) \left[   T^{1/q}\Vert y_1 - y_2 \Vert_{\mathrm{C}(\overline{J_T};X^\alpha)} + \Vert v_1 - v_2\Vert_{\mathrm{L}^q(J_T)} \right]\notag\\
			&\leq L(y) (1+T^{1/q}) \left[   \Vert y_1 - y_2 \Vert_{\mathrm{C}(\overline{J_T};X^\alpha)} + \Vert v_1 - v_2\Vert_{\mathrm{L}^q(J_T)} \right].\label{Lip_non_lin}
			\end{align}
			
			\item We show that $\tilde{F}$ is directionally differentiable. 
			
			To this aim, consider $y\in \mathrm{C}(\overline{J_T};X^\alpha)$ from Step 1~(b) and any $v\in \mathrm{L}^q(J_T)$.
			
			Let $(h,l)\in \mathrm{C}(\overline{J_T};X^\alpha)\times\mathrm{L}^q(J_T)$ be arbitrary and $\lambda_0>0$ small enough so that $y+\lambda h\in \overline{B_{\mathrm{C}(\overline{J_T};X^\alpha)}(y,\delta)}$ for all $\lambda\in (0,\lambda_0]$.
			
			For each $\lambda\in (0,\lambda_0]$ we define the differential quotient  
			\[
			\tilde{F}_\lambda:=\frac{1}{\lambda} [ \tilde{F}(y + \lambda h, v + \lambda l ) - \tilde{F}(y,v) ].
			\]
			
			For a.e. $s\in J_T$ we have that
			\[
			\lim\limits_{\lambda\rightarrow 0}\tilde{F}_\lambda(s) = f'[(y(s),v(s)); (h(s),l(s)) ]\in X
			\]
			because $f$ is directionally differentiable by Assumption~\ref{Ass:general_ass_and_short_notation}.
			
			We can also estimate for a.e. $s\in J_T$ and $\lambda_0$ small enough
			\[
			\Vert \tilde{F}_\lambda(s)\Vert_{X} \leq L(y) \left[ \|h(s)\|_{X^\alpha} + \vert l(s) \vert  \right]
			\]
			and the right side is contained in $\mathrm{L}^q(J_T)$.
			
			It follows by Lebesgue's dominated convergence theorem that $\tilde{F}_\lambda$ converges to the function
			\[s\mapsto f'[(y(s),v(s)); (h(s),l(s)) ]\]
			in $\mathrm{L}^q(J_T;X)$ as $\lambda \rightarrow 0$, which implies directional differentiability of $\tilde{F}$.
			
			This step is actually analogous to the proof of \cite[Lemma 3.1]{meyeroptimal}. The other steps needed some additional work.
			
			\item By Lemma \ref{Lem:hadam_for_lip}, Steps 1(b) and 1(c) imply that $\tilde{F}$ is Hadamard directionally differentiable and that $(h,l)\mapsto \tilde{F}'[(y,v); (h,l)]$ is Lipschitz continuous.
			
		\end{enumerate}
		\item Let $F: \mathrm{C}(\overline{J_T};X^\alpha) \rightarrow \mathrm{L}^q(J_T;X)$, $(F[y])(t):=f(y(t),\mathcal{W}[S y](t))$ be defined as in Theorem~\ref{Thm:state_equ_wellposed}.
		We show that $F$ is Hadamard directionally differentiable \cite[Lemma 3.1]{meyeroptimal}.
		
		Because the identity mapping $\mathrm{Id}$ on $\mathrm{C}(\overline{J_T};X^\alpha)$ and $S:\mathrm{C}(\overline{J_T};X^\alpha)\rightarrow \mathrm{C}(\overline{J_T})$ are linear and continuous they are Fréchet differentiable with derivatives $\mathrm{Id}$ and $S$.
		
		Lemma \ref{Lem:hadam_chain_rule} together with Subsection~\ref{SUBSEC:Hadam_Stop} yields that the mapping 
		\[
		y\mapsto (y,\mathcal{W}[Sy])
		\]
		is Hadamard directionally differentiable from $\mathrm{C}(\overline{J_T};X^\alpha)$ into $\mathrm{C}(\overline{J_T};X^\alpha)\times \mathrm{L}^q(J_T)$ with derivative
		\[
		h\mapsto (h,\mathcal{W}'[Sy;Sh]).
		\]
		
		Applying Lemma \ref{Lem:hadam_chain_rule} another time and using Step 1 we conclude that $F$ is Hadamard directionally differentiable with
		\begin{align*}
		F'[y;h](t)= f'[(y(t),\mathcal{W}[Sy](t)); (h(t),\mathcal{W}'[Sy;Sh](t))]
		\end{align*}
		for $y,h\in \mathrm{C}(\overline{J_T};X^\alpha)$ and a.e. $t\in J_T$.
		
		From Step 3 in the proof of Theorem~\ref{Thm:state_equ_wellposed} we know that $F$ is locally Lipschitz continuous from $\mathrm{C}(\overline{J_T};X^\alpha)$ to $\mathrm{C}(\overline{J_T},X)$ and therefore also from $\mathrm{C}(\overline{J_T};X^\alpha)$ to $\mathrm{L}^q(J_T;X)$.
		Lemma~\ref{Lem:hadam_for_lip} implies that for any $y\in \mathrm{C}(\overline{J_T};X^\alpha)$ the mapping $h\rightarrow F'[y;h]$ is Lipschitz continuous from $\mathrm{C}(\overline{J_T};X^\alpha)$ to $\mathrm{L}^q(J_T;X)$.
		
		\item We have seen in the end of Step 2 that $F$ is locally Lipschitz continuous from $\mathrm{C}(\overline{J_T};X^\alpha)$ to $\mathrm{C}(\overline{J_T},X)$ with a pointwise estimate of the form \eqref{Eq:Lipschitz_continuity_F}.
		
		\item
		We show that for any $y\in \mathrm{C}(\overline{J_T};X^\alpha)$ and $h\in \mathrm{L}^q(J_T;X)$ the integral equation
\[
\zeta(t)=\int\limits_{0}^t e^{-A_p(t-s)} [F'[y;\zeta](s)+h(s)] \,ds
\]
has a unique solution $\zeta(h)$ in $\mathrm{C}(\overline{J_T};X^\alpha)$ and that for fixed $y$ the mapping $h\mapsto \zeta(h)$ is Lipschitz continuous with a modulus $C=C(y,T)$.				
		
		By Step 2 the function $\zeta\mapsto F'[y;\zeta]$, where $F'[y;\zeta]$ is given by 
		
		$t\mapsto f'[(y(t),\mathcal{W}[Sy](t)); (\zeta(t),\mathcal{W}'[Sy;S\zeta](t))]
		$, is Lipschitz continuous from $\mathrm{C}(\overline{J_T};X^\alpha)$ to $\mathrm{L}^q(J_T;X)$.
		
		Similar to Theorem~\ref{Thm:state_equ_wellposed}, this together with \eqref{frac_pow_estimate1} and \eqref{integral_computaton} implies that for any $0<\tilde{T}\leq T$ the function 
\begin{align*}
g:\zeta \mapsto \left[ t\mapsto \int\limits_{0}^t e^{-A_p(t-s)} [F'[y;\zeta](s) + h(s)] \,ds \right]
\end{align*}
is well-defined on $\mathrm{C}([0,\tilde{T}];X^\alpha)$ and Lipschitz continuous with a modulus of the form
\[
L(\tilde{T})=C(y)e^{(1-\gamma)T}\tilde{T}^{1/q'-\alpha}.
\]
This observation together with Gronwall's Lemma already implies the statement about Lipschitz continuity for fixed $y$, provided that the fixed point mapping $h\mapsto \zeta(h)$ is well-defined.

We show by induction that $g$ has a fixed point in $\mathrm{C}(\overline{J_T};X^\alpha)$. Let $k\in \mathbb{N}$ be large enough so that $L\left(\frac{T}{k}\right)=C(y)e^{(1-\gamma)T}\left(\frac{T}{k}\right)^{1/q'-\alpha} <\frac{1}{2}$ and set
$t_j:=\frac{jT}{k}$ for $1\leq j \leq k$.

We prove that $g$ has a fixed point in $\mathrm{C}(\overline{J_{t_1}};X^\alpha)=\mathrm{C}([0,t_1];X^\alpha)$.

To this aim we define $H(t):=\int\limits_{0}^t e^{-A_p(t-s)}h(s)\,ds$ and $N_0:=\Vert H \Vert_{\mathrm{C}(\overline{J_T}; X^\alpha)}$ and consider $g$ on $\overline{B_{\mathrm{C}(\overline{J_{t_1}};X^\alpha)}(H,N_0)}$.
Note that $g$ is a contraction on $\mathrm{C}(\overline{J_{t_1}};X^\alpha)$ so that we can apply Banach's fixed point theorem if $g$ maps $\overline{B_{\mathrm{C}(\overline{J_{t_1}};X^\alpha)}(H,N_0)}$ into itself.
%\begin{align*}
%Y_0:=&\left\lbrace  z\in \mathrm{C}(\overline{J_{t_1}};X^\alpha): z(0)= 0, \right. \\ 
%&\left. \left\Vert z(t) - H(s)  \right\Vert_{\mathrm{C}(\overline{J_{t_1}},X^\alpha)} \leq \Vert H \Vert_{\mathrm{C}(\overline{J_T}, X^\alpha)} \right\rbrace.
%\end{align*}
%
%$Y_0$ is not empty since $H$ is in $Y_0$.

By definition we have $g(0)=H$.

Because $L\left(\frac{T}{k}\right)<\frac{1}{2}$, for $\zeta\in \overline{B_{\mathrm{C}(\overline{J_{t_1}};X^\alpha)}(H,N_0)}$ there holds
\begin{align*}
\|g(\zeta)(t)-H(t)\|_{X^\alpha}&= \|g(\zeta)(t)-g(0)(t)\|_{X^\alpha}\leq \frac{1}{2}\|\zeta(t)\|_{X^\alpha}\\
&\leq \frac{1}{2}\|\zeta(t)-H(t)\|_{X^\alpha} + \frac{1}{2}\|H(t)\|_{X^\alpha} \leq N_0
\end{align*}
so that indeed $g$ maps $\overline{B_{\mathrm{C}(\overline{J_{t_1}};X^\alpha)}(H,N_0)}$ into itself.

We obtain a unique fixed point $\zeta_1\in \overline{B_{\mathrm{C}(\overline{J_{t_1}};X^\alpha)}(H,N_0)}$ of $g_1:=g:\mathrm{C}(\overline{J_{t_1}};X^\alpha) \rightarrow \mathrm{C}(\overline{J_{t_1}};X^\alpha)$.

Inductively, we set $N_j:=2N_{j-1}+ N_0$ for $2\leq j \leq k$ and define

$g_j: \mathrm{C}(\overline{J_{t_j}},X^\alpha) \rightarrow \mathrm{C}(\overline{J_{t_j}},X^\alpha)$ as
\[
g_j(\zeta)(t) := \left\lbrace \begin{matrix}
\zeta_{j-1}(t) & \text{if }t\in [0,t_{j-1}],\\
\zeta_{j-1}(t_{j-1}) + \int\limits_{t_{j-1}}^t e^{-A_p(t-s)} [ F'[y;\zeta](s) + h(s)]\,ds & \text{if }t\in [t_{j-1},t_j]
\end{matrix}\right.,
\]
assuming that the unique fixed point $\zeta_{j-1}$  of $g_{j-1}$ exists from the previous step.
We show that $g_j$ has a fixed point.
%\begin{align*}
%Y_j:=&\left\lbrace  z\in \mathrm{C}(\overline{J_{t_j}},X^\alpha): z(t_{j-1})= z_{j-1}(t_{j-1}), \right. \\ 
%&\left. \left\Vert z - H  \right\Vert_{\mathrm{C}(\overline{J_{t_j}},X^\alpha)} \leq N_j\Vert H \Vert_{\mathrm{C}(\overline{J_T}, X^\alpha)} \right\rbrace.
%\end{align*}

Note that $g_j(0)=\zeta_{j-1}(t_{j-1}) + H-H(t_{j-1})\in \mathrm{C}(\overline{J_{t_j}},X^\alpha)$ and that $g_j$ is a $\frac{1}{2}$-contraction on $\mathrm{C}(\overline{J_{t_j}};X^\alpha)$.
So we are left to show that $g_j$ maps $\overline{B_{\mathrm{C}(\overline{J_{t_j}};X^\alpha)}(H,N_j)}$ into itself.

Let $\zeta\in \overline{B_{\mathrm{C}(\overline{J_{t_j}};X^\alpha)}(H,N_j)}$ be given.
On $[0,t_{j-1}]$ we can estimate 
\[
\left\Vert g_j(\zeta)(t) - H(t)  \right\Vert_{X^\alpha} = \|\zeta_{j-1}(t) - H(t) \|\leq N_{j-1}\leq N_{j}
\]
by induction.
For $t\in [t_{j-1},t_j]$ we can estimate
\begin{align*}
\|g_j(\zeta)(t)-H(t)\|_{X^\alpha}&= \|\zeta_{j-1}(t_{j-1})-H(t_{j-1}) + g_j(\zeta)(t)-g_j(0)(t)\|_{X^\alpha}\\
&\leq \|\zeta_{j-1}(t_{j-1})-H(t_{j-1})\|_{X^\alpha} + \frac{1}{2}\|\zeta(t)\|_{X^\alpha}\\
&\leq N_{j-1} + \frac{1}{2}\|\zeta(t)-H(t)\|_{X^\alpha} + \frac{1}{2}\|H(t)\|_{X^\alpha}\\
&\leq N_{j-1}+\frac{N_j}{2} + \frac{1}{2}\Vert H \Vert_{\mathrm{C}(\overline{J_T}, X^\alpha)}\\
&= N_{j-1}+\frac{2N_{j-1}+\Vert H \Vert_{\mathrm{C}(\overline{J_T}; X^\alpha)}}{2} + \frac{1}{2}\Vert H \Vert_{\mathrm{C}(\overline{J_T}, X^\alpha)}\\
&= 2N_{j-1}+\Vert H \Vert_{\mathrm{C}(\overline{J_T}, X^\alpha)}=N_j.
\end{align*}
So indeed $g_j$ maps $\overline{B_{\mathrm{C}(\overline{J_{t_j}};X^\alpha)}(H,N_j)}$ into itself and we obtain a unique fixed point

$\zeta_j\in \overline{B_{\mathrm{C}(\overline{J_{t_j}};X^\alpha)}(H,N_j)}$ of $g_j$.

We have
\[
\zeta_2(t) = \zeta_1(t)= g(\zeta_1)(t)\]
for $t\in\overline{J_{t_1}}$ and
\[
\zeta_2(t) = \int\limits_{0}^{t_1} e^{-A_p(t-s)} [ F'[y;\zeta_1](s) +h(s)]\,ds + \int\limits_{t_1}^t e^{-A_p(t-s)} [ F'[y_0;\zeta_2](s) +h(s)]\,ds
\]
for $t\in[t_1,t_2]$ which implies
\[
\zeta_2(t)=\int\limits_{0}^t e^{-A_p(t-s)} [ F'[y;\zeta_2](s) + h(s)] \,ds = g(\zeta_2)(t)
\]
on $[0,t_2]$.

Inductively, it follows $\zeta_j=g(\zeta_j)$ for all $j\in\{1,\cdots k\}$ which shows that $\zeta=\zeta(h):=\zeta_k$ is the unique solution of the integral equation
\[
\zeta(t)=\int\limits_{0}^t e^{-A_p(t-s)} [F'[y;\zeta](s)+h(s)] \,ds
\]
in $\mathrm{C}(\overline{J_T}, X^\alpha)$, i.e. a fixed point of $g$.

		\item
		We now come to the proof of the statement of the theorem.
				
		Let any $u\in \mathrm{L}^q(J_T;X)$ be given and $y=G(u)$.
		For $h\in \mathrm{L}^q(J_T;X)$ and $\lambda>0$ we denote $y_\lambda :=G(u+\lambda h)$.
		Let $\zeta=\zeta(h)\in \mathrm{C}(\overline{J_T};X^\alpha)$ be the function from Step 4.
		
		Similar as in \cite[Theorem 3.2]{meyeroptimal} we estimate with \eqref{frac_pow_estimate1} and \eqref{integral_computaton}, Step 2 and Step 3 and for $\lambda>0$ small enough
		\begin{align*}
		&\left\|\frac{y_\lambda(t) - y(t)}{\lambda} - \zeta(t) \right\|_{X^\alpha}\\
		&\leq  C_\alpha e^{(1-\gamma)T}\int\limits_{0}^t (t-s)^{-\alpha} \left(
		\left\|\frac{ (F[y + \lambda \zeta])(s) - (F[y])(s)}{\lambda} - F'[y;\zeta](s)  \right\|_X \right.\\
		&\left.
		+
		\left\|\frac{ (F[y + \lambda \zeta])(s) - (F[y_\lambda])(s)}{\lambda}  \right\|_X
		\right) \,ds\\		
		&\leq c e^{(1-\gamma)T}
		\left( 
			t^{1/q'-\alpha} 	\left\|\frac{ F[y + \lambda \zeta] - F[y]}{\lambda} - F'[y;\zeta]  \right\|_{\mathrm{L}^q(J_T;X)}
		\right.\\
		&\left. 
		+ L(y) \int\limits_{0}^t (t-s)^{-\alpha}\sup_{0\leq \tau' \leq s}\left\|\frac{y_\lambda(\tau') - y(\tau')}{\lambda} - \zeta(\tau')  \right\|_{X^\alpha} \,ds\right).
		\end{align*}
		The first term converges to zero with $\lambda\rightarrow 0$ by Step 2.		
		The estimate of the second term holds because of \eqref{Eq:Lipschitz_continuity_F} in Step 3, which was the local Lipschitz continuity of $F$, and by local Lipschitz continuity of $G$ from $\mathrm{L}^q(J_T;X)$ to $\mathrm{C}(\overline{J_T};X^\alpha)$ (see Theorem~\ref{Thm:state_equ_wellposed}).
		
		We take the supremum $\sup_{0\leq \tau \leq t}$ on both sides and apply
		Gronwall's Lemma to see that $\frac{y_\lambda - y}{\lambda}$ converges to $\zeta$ in $\mathrm{C}(\overline{J_T};X^\alpha)$.
		So we find that $\zeta$ is the directional derivative of $G$ at $u$ in direction $h$.
		
		Local Lipschitz continuity of $G$ and Lemma~\ref{Lem:hadam_for_lip} imply that the solution mapping for problem \eqref{state_eqation_abstract} is Hadamard directionally differentiable from $\mathrm{L}^q(J_T;X)$ to $\mathrm{C}(\overline{J_T};X^\alpha)$.
		
		The statement for $Y_{q,0}$ follows with Remark~\ref{Rem:maximal parabolic regularity} just as in the proof of \cite[Theorem 3.2]{meyeroptimal}.
	\end{enumerate}
\end{proof}

\section{Application to an optimal control problem}\label{Sec:Optimal_contr}
In this section, we apply the results from Theorem~\ref{Thm:state_equ_wellposed} and Theorem~\ref{Thm:Solution_op_Hadamard} to an optimal control problem. We consider either
distributed controls in
\begin{align*}
U_1:= \mathrm{L}^2 \left(J_T;\tilde{U}_1\right) := \mathrm{L}^2 \left(J_T;[\mathrm{L}^2(\Omega)]^m\right)
\end{align*}
or Neumann boundary controls in
\begin{align*}
U_2:= \mathrm{L}^2 \left(J_T;\tilde{U}_2\right):= \mathrm{L}^2 \left(J_T; \prod_{i=1}^m \mathrm{L}^2(\Gamma_{N_i},\mathcal{H}_{d-1})\right).
\end{align*}
Moreover, we will define continuous operators
$B_i : \tilde{U}_i \rightarrow X$ for $i\in \{1,2\}$, see Assumption~\ref{Ass:General_for_control}.

%In Section~\ref{Sec:Well_posed} we show well-posedness of the reaction-diffusion system with 
%
%$u\in \mathrm{L}^q((0,T);\mathbb{W}_{\Gamma_D}^{-1,p}(\Omega))$, which also yields global existence for the state equation.

With $u\in U_i$, Theorem~\ref{Thm:state_equ_wellposed} implies well-posedness of the following state equation:
\begin{align}
\dot{y}(t) +A_p y(t) &= f(y(t),z(t))  + B_i u(t)& \hspace{0.75cm} &\text{in } \mathbb{W}_{\Gamma_D}^{-1,p}(\Omega) \text{ for }t\in (0,T), \label{state_equ_y}\\
y(0)&= 0 && \text{in } \mathbb{W}_{\Gamma_D}^{-1,p}(\Omega),\notag\\
(\dot{z}(t)-S\dot{y}(t))(z(t)-\xi) &\leq 0 && \text{for } \xi\in [a,b] \text{ and } t\in (0,T),\label{state_equ_z}\\
z(t)&\in [a,b] && \text{for } t\in[0,T],\notag\\
z(0)&=z_0\notag.
\end{align} 
Note that \eqref{state_equ_z} implies $z=\mathcal{W}[Sy]$.
For $i\in\{1,2\}$ and given $\kappa>0$ consider the optimal control problem
\begin{align}
\min_{u \in U_i} J(y,u) &:= \frac{1}{2}  \|y-y_d\|_{U_1}^2+ \frac{\kappa}{2}\| u \|_{U_i}^2\notag\\
&= \frac{1}{2}  \int_{0}^{T}\|y(s)-y_d(s)\|_{[\mathrm{L}^2(\Omega)]^m}^2\, ds + \frac{\kappa}{2}\int_{0}^{T}\|u(s)\|_{ \tilde{U}_i}^2\, ds\label{opt_control_ control_problem}
\end{align} 
subject to \eqref{state_equ_y}, \eqref{state_equ_z}.

\begin{assumption}\label{Ass:General_for_control}
	In addition to Assumption~\ref{Ass:General_for_Hadam} we assume:
	\begin{itemize}
		\item $\alpha\in \left(0,\frac{1}{2}\right)$. This assumption is needed in the proof of Lemma~\ref{Lem:convergence_nonreg}.
		\item $B_1$ is defined by
		\[
		B_1:[\mathrm{L}^2(\Omega)]^m\rightarrow X,\ \langle B_1 u,v \rangle_{\mathbb{W}_{\Gamma_D}^{1,p'}(\Omega)}:=\int_\Omega u\cdot v\,dx, \ v\in\mathbb{W}_{\Gamma_D}^{1,p'}(\Omega).
		\]
		Since $2\geq p\left(1-\frac{1}{d}\right)$ the embeddings 
		$\mathrm{L}^2(\Gamma_{N_j},\mathcal{H}_{d-1})\hookrightarrow \mathrm{W}_{\mathrm{\Gamma_D}_j}^{-1,p}(\Omega)$ are continuous for $j\in \{1,\cdots,m\}$ \cite[Remark 5.11]{rehbergsystems}.
		
		Therefore also
		\[B_2:\prod_{j=1}^m \mathrm{L}^2(\Gamma_{N_j},\mathcal{H}_{d-1})\rightarrow X,\ \langle B_2y,v \rangle_{\mathbb{W}^{1,p'}(\Omega)}=\sum_{j=1}^m \int_{\Gamma_{N_j}}y_jv_j \,d\mathcal{H}_{d-1},\ \ v\in\mathbb{W}_{\Gamma_D}^{1,p'}(\Omega)\]
		is continuous.
%		\item We write
%		\begin{align*}
%		U_1&= \mathrm{L}^2 \left(J_T;[\mathrm{L}^2(\Omega)]^m\right)
%		\text{ and }
%		U_2= \mathrm{L}^2 \left(J_T; \prod_{j=1}^m \mathrm{L}^2(\Gamma_{N_j},\mathcal{H}_{d-1})\right).
%		\end{align*}		
		\item The desired state $y_d$ in \eqref{opt_control_ control_problem} is in $U_1$ and $\kappa>0$ is given.
	\end{itemize}
\end{assumption}

\begin{remark}	
Theorem~\ref{Thm:Solution_op_Hadamard} yields Hadamard directional differentiability of $G\circ B_i: U_i \rightarrow Y_{2,0}$ for $i\in\{1,2\}$ and $(y,z)=(G(B_i u ), \mathcal{W}[SG(B_i u)])$ solves \eqref{state_equ_y}, \eqref{state_equ_z} for $u\in U_i$.
Therefore the reduced cost function $\mathcal{J}:U_i\rightarrow \mathbb{R}$, $\mathcal{J}(u)=J(G(B_iu),u)$ is Hadamard directionally differentiable.
\end{remark}

\begin{lemma}\label{Lem:convergence_nonreg}
	Let Assumption~\ref{Ass:General_for_control} hold.
	
	Suppose that for $\{u_n\}_{n\in\mathbb{N}}\subset U_i$ it holds $u_n \rightharpoonup u$ in $U_i$ with $i\in \{1,2\}$.
	
	Then $y_n=G(B_iu_n)\rightarrow G(B_iu)$ weakly in $Y_{2,0}$ and strongly in $\mathrm{C}(\overline{J_T};X^\alpha)$ and
	
	$z_n=\mathcal{W}[Sy_n]\rightarrow \mathcal{W}[SG(B_iu)]$ weakly in $\mathrm{H}^1(J_T)$ and strongly in $\mathrm{C}(\overline{J_T})$ \cite[Lemma 2.3]{brokate2013optimal}. 
	
	If the convergence of $u_n$ is strong then $y_n\rightarrow G(B_iu)$ in $Y_{2,0}$ strongly.
\end{lemma}
%\begin{lemma}
%	\[G:\mathrm{L}^2(J_T;X) \rightarrow Y_{2,0}\]
%	is weakly continuous
%	\cite[Lemma 2.10]{meyeroptimal}.
%\end{lemma}
\begin{proof}
	The proof is a combination of the proofs for \cite[Lemma 2.10]{meyeroptimal} and \cite[Lemma 2.3]{brokate2013optimal}.
	
	Let $u_n\rightharpoonup u$ in $U_i$.
	
	By Assumption~\ref{Ass:General_for_control} we have $\alpha \in (0,\frac{1}{2})$ so that $\frac{1}{1-\alpha}<2=q$. We can therefore use Theorem~\ref{Thm:state_equ_wellposed} and Theorem~\ref{Thm:Solution_op_Hadamard} with $u$ and $h$ replaced by $B_iu$ and $B_ih$ and with $\mathrm{L}^2(J_T;X)$ replaced by $U_i$. 
	By Remark~\ref{Rem:maximal parabolic regularity} and \eqref{solution_op_bdd} there exists some $c>0$ such that
	\begin{align*}
	\|y_n\|_{Y_{2,0}}&\leq \left\|\left(\frac{d}{dt} + A_p\right)^{-1}\right\|_{\mathcal{L}(\mathrm{L}^2(J_T;X),Y_{2,0})}(\|B_iu_n\|_{\mathrm{L}^2(J_T;X)} + \|F[y_n]\|_{\mathrm{L}^2(J_T;X)})\\ &\leq c(1+\|B_iu_n\|_{\mathrm{L}^2(J_T;X)})
	\end{align*}
	so that a subsequence $y_{n_k}$ weakly converges in $Y_{2,0}$ to some $y$.
	By Remark~\ref{Rem:embeddings} we know that $Y_{2,0}$ is compactly embedded into $\mathrm{C}(\overline{J_T};X^\alpha)$ so that the convergence is strong in this space.
	We also have that $Sy_{n_k}$ converges weakly to $Sy$ in  $\mathrm{H}^1(J_T)$ because $S\in X^*$.
	
	From Subsection~\ref{SUBSEC:Stop} we know that $\mathcal{W}$ is weakly continuous on $\mathrm{H}^1(J_T)$ so that weak convergence of $Sy_{n_k}$ implies weak convergence of $z_{n_k}$ to 
	$\mathcal{W}[Sy]=z$ in $\mathrm{H}^1(J_T)$ and then also strong convergence in $\mathrm{C}(\overline{J_T})$.
%	Hence for yet another subsequence again denoted by the same symbol we have that $z_{n_k}$ converges weakly in $\mathrm{H}^1(J_T)$ and therefore strongly in $\mathrm{C}(\overline{J_T})$ to some $z$.
%	Because $[a,b]$ is close we know that $z(t)\in [a,b]$ for all $t\in \overline{J_T}$.
%	
%	Moreover, as in the proof of \cite[Lemma 2.3]{brokate2013optimal} we observe that we may pass to the limit in the integral equation 
%	\begin{align*}
%	\int_{0}^{T}(\dot{z}_{n_k}(t) - Sy_{n_k}(t))(z_{n_k}(t)-\xi(t)\, dt
%	\end{align*}
%	where $\xi\in \mathrm{L}^2J_T$.
%	Since for a.e. $t$ with $\xi(t) \in [a,b]$ it is
%	\[(\dot{z}_{n_k}(t) - Sy_{n_k}(t))(z_{n_k}(t)-\xi(t)\leq 0,\]
%	in the limit it is
%	\[(\dot{z} - Sy)(z-\xi\leq 0\]
%	for a.e. $t\in J_T$ and all $x\in [a,b]$, which implies $z=\mathcal{W}[Sy]$.
	
	Weak continuity of $\frac{d}{dt}$, $A_p$ and $B_i$ yields
	\[\frac{d}{dt}y_{n_k} + A_p y_{n_k} \rightharpoonup \frac{d}{dt}y + A_p y \text{ and } B_i u_{n_k} \rightharpoonup B_i u \text{ in } \mathrm{L}^2(J_T;X)\]
	\cite[Lemma 2.10]{meyeroptimal}.
	
	For $n_k$ large enough we obtain by strong convergence of $y_{n_k}$ in $\mathrm{C}(\overline{J_T};X^\alpha)$ and by local Lipschitz continuity of $f$
	\begin{align*}
	\|f(y_{n_k},z_{n_k})-f(y,z)\|_{\mathrm{C}(\overline{J_T};X)}\leq L(y)(\|y_{n_k}-y\|_{\mathrm{C}(\overline{J_T};X^\alpha)}+ \|z_{n_k}-z\|_{\mathrm{C}(\overline{J_T})})
	\end{align*}
	so that $f(y_{n_k}(\cdot),z_{n_k}(\cdot))$ converges to $f(y(\cdot),z(\cdot))$ in $\mathrm{C}(\overline{J_T};X)$.
	
	We pass to the limit in \eqref{state_eqation_abstract} and conclude that $y=G(B_iu)$ and $z=\mathcal{W}[Sy]$.
	Uniqueness of the limit implies (weak) convergence of the whole sequence.
	
	The statement about strong convergence if $\{u_n\}$ converges to $u$ strongly in $U_i$ follows because in this case
	\begin{align}
	\|y_n - y\|_{Y_{2,0}}
	&\leq \left\|\left(\frac{d}{dt} + A_p\right)^{-1}\right\|_{\mathcal{L}(\mathrm{L}^2(J_T;X),Y_{2,0})}
	\left(
		\|B_i(u_n-u)\|_{\mathrm{L}^2(J_T;X)} + \|F[y_n] - F[y] \|_{\mathrm{L}^2(J_T;X)}
	\right)\label{estim_Lem:convergence_nonreg}
	\end{align} 
	and since the right side then converges to zero.
\end{proof}

\begin{theorem}\label{Thm:exist_opt_control}
	Let Assumption~\ref{Ass:General_for_control} hold.
	Then for $i\in \{1,2\}$, there exists an optimal control $\overline{u}\in U_i$ for the optimal control problem \eqref{state_equ_y}-\eqref{opt_control_ control_problem}. This means that $\overline{u}$, together with the optimal state $\overline{y}=G(\overline{u})$, which solves \eqref{state_equ_y}, are a solution of the minimization problem \eqref{opt_control_ control_problem}. The solution of \eqref{state_equ_z} is given by $\overline{z}=\mathcal{W}[S\overline{y}]$.
\end{theorem}

\begin{proof}
The proof uses Lemma~\ref{Lem:convergence_nonreg} and is analogous to the proof of \cite[Proposition 2.11]{meyeroptimal}.
\end{proof}

\begin{remark}
	In a forthcoming paper we derive an adjoint system and optimality conditions for problem \eqref{state_equ_y}-\eqref{opt_control_ control_problem}.
	The differences between the control problem for $U_1$ and $U_2$ will become obvious during this analysis.
	We will first derive optimality conditions for problem \eqref{state_equ_y}-\eqref{opt_control_ control_problem} with either distributed or boundary controls, i.e. $i\in\{1,2\}$.
	Since $B_1$ has dense range we are able to improve those for $i=1$. We can also show uniqueness of the adjoint system for the case of distributed controls.
\end{remark}

\section*{Acknowledgement}
The author is supported by the DFG through the International Research Training Group IGDK 1754 „Optimization and Numerical Analysis for Partial Differential Equations with Nonsmooth Structures".
The author would like to thank Prof. Brokate from the Technical University of Munich and Prof. Fellner from the Karl-Franzens University of Graz for thoroughly proofreading the manuscript, as well as Dr. Joachim Rehberg from the Weierstrass Institute in Berlin for the helpful discussions. 

\newpage
\bibliography{Literatur.bib}
\bibliographystyle{plain}
\end{document}